\documentclass[a4paper,11pt,reqno]{article}

\usepackage[utf8]{inputenc}
\usepackage[T1]{fontenc}
\usepackage{lmodern}
\usepackage[english]{babel}
\usepackage{microtype}

\usepackage{amsmath,amssymb,amsfonts,amsthm}
\usepackage{mathtools,accents}
\usepackage{mathrsfs}
\usepackage{aliascnt}
\usepackage{braket}
\usepackage{bm}

\usepackage[a4paper,margin=3cm]{geometry}
\usepackage[citecolor=blue,colorlinks]{hyperref}

\usepackage{enumerate}
\usepackage{xcolor}

\makeatletter
\g@addto@macro\@floatboxreset\centering
\makeatother

\makeatletter
\def\newaliasedtheorem#1[#2]#3{
  \newaliascnt{#1@alt}{#2}
  \newtheorem{#1}[#1@alt]{#3}
  \expandafter\newcommand\csname #1@altname\endcsname{#3}
}
\makeatother

\numberwithin{equation}{section}

\newtheoremstyle{slanted}{\topsep}{\topsep}{\slshape}{}{\bfseries}{.}{.5em}{}

\theoremstyle{plain}
\newtheorem{theorem}{Theorem}[section]
\newaliasedtheorem{proposition}[theorem]{Proposition}
\newaliasedtheorem{lemma}[theorem]{Lemma}
\newaliasedtheorem{corollary}[theorem]{Corollary}
\newaliasedtheorem{conjecture}[theorem]{Conjecture}
\newaliasedtheorem{counterexample}[theorem]{Counterexample}

\theoremstyle{definition}
\newaliasedtheorem{definition}[theorem]{Definition}
\newaliasedtheorem{question}[theorem]{Question}
\newaliasedtheorem{example}[theorem]{Example}

\theoremstyle{remark}
\newaliasedtheorem{remark}[theorem]{Remark}

\let\altphi\phi
\let\phi\varphi
\let\varphi\altphi
\let\altphi\undefined

\newcommand{\di}{\mathop{}\!\mathrm{d}}

\DeclareMathOperator{\supp}{supp}

\newcommand{\Ch}{{\sf Ch}}

\newcommand{\dist}{\mathsf{d}}

\newcommand{\meas}{\mathfrak{m}}

\DeclareMathOperator{\RCD}{RCD}

\newfont{\tmpf}{cmsy10 scaled 2500}

\newcommand{\intav}{{\mathop{\int\kern-10pt\rotatebox{0}{\textbf{--}}}}}

\renewcommand{\ }{\text{ }}
\def\<{\langle}
\def\>{\rangle}

\begin{document}
\title{Spectral distances on RCD spaces}
\author{
Shouhei Honda
\thanks{Mathematical Institute, Tohoku University, \url{shouhei.honda.e4@tohoku.ac.jp} Current address: Graduate School of Mathematical Sciences, The University of Tokyo, \url{shouhei.@ms.u-tokyo.ac.jp} Keywords: Ricci curvature, Laplacian, metric measure space, eigenfunction, heat kernel, and Gromov-Hausdorff convergence. MSC: 53C20, 53C21 and 53C23.}} 
\maketitle
\begin{abstract}
We provide relationships between the spectral convergences in B\'erard-Besson-Gallot sense, in Kasue-Kumura sense and the measured Gromov-Hausdorff convergence, for compact finite dimensional $\RCD$ spaces. As an independent interest, a canonical (spectral) approximation map between such spaces constructed by given spectral data is obtained.
\end{abstract}

\section{Introduction}
The paper establishes \textit{reconstructions} of compact spaces with Ricci curvature bounded from below via embeddings by eigenfunctions/heat kernels. Moreover related convergence notions are discussed. The precise descriptions are explained below.
\subsection{Spectral convergence}
Roughly speaking, \textit{spectral convergence} means that spectral information, including eigenvalues/eigenfunctions, behave continuously with respect to  a convergence of spaces in some topology. Since we have a couple of choices on the topologies, the purpose of this subsection is to clarify them.
\subsubsection{B\'erard-Besson-Gallot sense}
Let $(M^n, g)$ be an $n$-dimensional closed Riemannian manifold. Then the spectrum of the operator $-\Delta^g=-\mathrm{tr}(\mathrm{Hess})$, counted with multiplicities, can be written as follows;
\begin{equation}
0=\lambda_0^g<\lambda^g_1\le \lambda^g_2 \le \cdots \to \infty.
\end{equation}
We say that a sequence of smooth functions on $M^n$, $\mathbf{a}=(\phi_0^{\mathbf{a}}, \phi_1^{\mathbf{a}},\ldots)$, is a \textit{spectral datum} of $(M^n ,g)$ if $\Delta^g\phi_i^{\mathbf{a}}+\lambda_i^{\mathbf{a}}\phi_i^{\mathbf{a}}=0$ and $(\phi_i^{\mathbf{a}}, \phi_j^{\mathbf{a}})_{L^2}=\delta_{ij}$ are satisfied (namely $\{\phi_i^{\mathbf{a}}\}_i$ gives an orthonormal basis of $L^2(M^n, \mathrm{vol}^g)$), where $\mathrm{vol}^g$ denotes the Riemannian volume measure on $M^n$ defined by the Riemannian metric  $g$.

B\'erard-Besson-Gallot discussed in \cite{BerardBessonGallot} a family of smooth maps $I_t^{\mathbf{a}}:M^n \to \ell^2$ for all spectral datum $\mathbf{a}$ and $t \in (0, \infty)$ defined by
\begin{equation}
I^{\mathbf{a}}_t(x):=\left( \sqrt{\mathrm{vol}^gM^n}e^{-\lambda_i^gt/2}\phi_i^{\mathbf{a}}(x)\right)_{i \ge 1},
\end{equation}
which gives indeed an embedding of $M^n$ into $\ell^2$. More precisely they defined a pseudo distance by using these embeddings, denoted by $\dist_{\mathrm{Spec}}^t((M^n, g), (N^k, h))$, between closed Riemannian manifolds $(M^n, g)$ and $(N^k, h)$ as follows; 
\begin{align}
&\dist_{\mathrm{Spec}}^t\left((M^n, g), (N^k, h)\right) \nonumber \\
&:=\max \left\{ \sup_{\mathbf{a}}\inf_{\mathbf{b}}\dist_{\ell^2}^{\mathrm{H}}\left(I^{\mathbf{a}}_t(M^n), I^{\mathbf{b}}_t(N^k) \right), \quad \sup_{\mathbf{b}}\inf_{\mathbf{a}}\dist_{\ell^2}^{\mathrm{H}}\left(I^{\mathbf{a}}_t(M^n), I^{\mathbf{b}}_t(N^k) \right)\right\},
\end{align}
where $\dist_{\ell^2}^{\mathrm{H}}$ denotes the Hausdorff distance in $\ell^2$. A main result in \cite{BerardBessonGallot} states that $\dist_{\mathrm{Spec}}^t$ is actually a distance in the smooth framework, namely it is trivially symmetric, it satisfies the triangle inequality and it is nondegenerate. The nontrivial part is the nondegeneracy, namely if $\dist_{\mathrm{Spec}}^t((M^n,g), (N^k, h))=0$ holds, then $(M^n, g)$ is isometric to $(N^k, h)$.
Let us recall the proof of this nondegeneracy as follows.

Firstly, by compactness, there exist spectral data $\mathbf{a}$ and $\mathbf{b}$ such that $I^{\mathbf{a}}_t(M^n)=I^{\mathbf{b}}_t(N^k)$, which allows us to find a homeomorphism $f:M^n \to N^k$ preserving the embeddings $I^{\mathbf{a}}_t, I^{\mathbf{b}}_t$ (in particular $n=k$). Secondly 
 we observe that any point has a chart around the point by the restriction of an eigenmap. Combining this with beautiful arguments proves that $f$ gives a diffeomorphism and that $f$ preserves the Laplacians, namely
 \begin{equation}\label{p}
 \Delta^g (\psi \circ f)= (\Delta^h \psi)\circ f,\quad \forall \psi \in C^{\infty}(N^k).
 \end{equation}
Since  (\ref{p}) implies  by looking at the principal symbols that $f$ also preserves the Riemannian metrics, we conclude that $f$ is an isometry.

Roughly speaking, this nondegeneracy says that any spectral datum reconstructs the metric structure of a closed Riemannian manifold. 

 It is emphasized that a precompactness result with respect to $\dist^t_{\mathrm{Spec}}$ is given in the same paper and that a Lipschitz (more strongly, a smooth) convergence is \textit{not} enough to get the convergence with respect to $\dist^t_{\mathrm{Spec}}$, even with flat curvature. In fact under a smooth convergence;
\begin{equation}\label{flat}
\left( \mathbb{S}^1(1)\times \mathbb{S}^1(1+\epsilon), g_{\mathbb{S}^1(1)}\oplus g_{\mathbb{S}^1(1+\epsilon)}\right) \to \left( \mathbb{S}^1(1)\times \mathbb{S}^1(1), g_{\mathbb{S}^1(1)}\oplus g_{\mathbb{S}^1(1)}\right) 
\end{equation}
as $\epsilon \to 0$ with $\epsilon \neq 0$, the corresponding spectral distances $\dist_{\mathrm{Spec}}^t$ between the LHS and the RHS in (\ref{flat}) do not converge to $0$ for any $t \in (0, \infty)$, where $g_{\mathbb{S}^1(r)}$ denotes the canonical Riemannian metric on $\mathbb{S}^1(r):=\{x \in \mathbb{R}^2; |x|_{\mathbb{R}^2}=r\}$. See \cite[Example 28]{BerardBessonGallot}. This tells us that $\dist_{\mathrm{Spec}}^t$-convergence might be completely different from the previously known convergence notions in metric geometry. It is worth mentioning that \cite[Proposition 16]{BerardBessonGallot} provides us collapsing examples of $\dist_{\mathrm{Spec}}^t$-convergence.

Here let us state the following natural questions.
\begin{enumerate}
\item[(Q$1$)] Is it possible to generalize the spectral distance $\dist_{\mathrm{Spec}}^t$ above to a nonsmooth setting?
\item[(Q$2$)] Can we give a necessary and sufficient condition for the validity of $\dist_{\mathrm{Spec}}^t$-convergence in terms of the measured Gromov-Hausdorff (mGH) convergence?
\end{enumerate}
As mentioned above, only considering mGH convergence is not enough in order to give a positive answer to (Q$2$) because the Lipschitz convergence is strictly stronger than the mGH-convergence.
To give a positive answer to (Q$1$), in the paper, we adopt \textit{metric measured spaces with Ricci curvature bounded from below}, so called $\RCD$ \textit{spaces}, as a nonsmooth context. 
The first main results, Theorems \ref{thmnond} and \ref{mainthm2}, giving positive answers to both (Q$1$) and (Q$2$), will be explained later. 

\subsubsection{Kasue-Kumura sense}
In \cite{KK, KK2}, Kasue-Kumura defined a distance between weighted closed Riemannian manifolds by using their heat kernels. The idea is close to that of \cite{BerardBessonGallot}, although different. The precise definition is as follows; for all weighted closed Riemannian manifolds $(M^n, g, e^{-\psi}\di \mathrm{vol}^g), (N^k, h, e^{-\phi}\di \mathrm{vol}^h)$, denoting by $p, q$ their heat kernels, respectively, the \textit{spectral distance} in the sense of Kasue-Kumura between them is defined by the infimum of $\epsilon \in (0, \infty)$ satisfying that there exist maps $f_1:M^n \to N^k, f_2:N^k \to M^n$ such that 
\begin{equation}
e^{-(t+1/t)}\left|q(f_1(x), f_1(\tilde x), t)-p(x, \tilde x, t)\right|<\epsilon,\quad \forall x,\,\,\forall \tilde x \in M^n,\,\,\,\forall t \in (0, \infty)
\end{equation}
and 
\begin{equation}
e^{-(t+1/t)}\left|p(f_2(y), f_2(\tilde y), t)-q(y, \tilde y, t)\right|<\epsilon,\quad \forall y,\,\,\forall\tilde y \in N^k,\,\,\,\forall t \in (0, \infty)
\end{equation}
are satisfied. Then they also established a precompactness result with respect to this distance.
Here let us give the following natural question as in the case of B\'erard-Besson-Gallot.
\begin{enumerate}
\item[(Q$3$)] After generalizing the spectral distance in the sense of Kasue-Kumura, denoted by $\tilde \dist_{\mathrm{Spec}}$ in the paper, to a nonsmooth setting, can we give a necessary and sufficient condition for the validity of $\tilde \dist_{\mathrm{Spec}}$-convergence in terms of the mGH convergence?
\end{enumerate}
The second main result, Theorem \ref{kk}, provides a positive answer to (Q$3$) as explained in the next subsection.
\subsection{Main results}
We say that a metric measured space $\mathbb{X}=(X, \dist_X, \meas_X)$ is said to be an $\RCD(K, N)$ \textit{space} for some $K \in \mathbb{R}$ and some $N \in [1, \infty)$, or $\RCD$ space for short, if its Ricci curvature is bounded below by $K$ and the dimension is bounded above by $N$ in a synthetic sense, and the $H^{1,2}$-Sobolev space is a Hilbert space. 
See subsection \ref{rcddef} for a brief introduction on the definition. In the sequel all spaces we will discuss are assumed to be not single points.

Let us fix a compact $\RCD(K, N)$ space $\mathbb{X}=(X, \dist_X, \meas_X)$. Then the spectrum of the operator $-\Delta_X$ on $X$ can be also written by
\begin{equation}\label{eq:eignval}
0=\lambda_0(\mathbb{X})<\lambda_1(\mathbb{X}) \le \lambda_2(\mathbb{X}) \le \cdots \to \infty,
\end{equation}
counted with (finite) multiplicities, because the canonical inclusion of the $H^{1,2}$-Sobolev space into $L^2(X, \meas_X)$ is a compact operator by the compactness of $X$ (see subsection \ref{spectraltheory} or \cite[Appendix]{Honda3}).
Thus by similar ways as in the case of closed Riemannian manifolds, we can also define a spectral datum $\mathbf{a}$ of $\mathbb{X}$, a family of topological embeddings $I^{\mathbf{a}}_t:X \to \ell^2$ by
\begin{equation}
I^{\mathbf{a}}_{t}(x):=\left(\sqrt{\meas_X (X)}e^{-\lambda_i(\mathbb{X})t/2}\phi_i^{\mathbf{a}}(x)\right)_{i \ge 1}
\end{equation}
and the spectral pseudo distance $\dist_{\mathrm{Spec}}^t(\mathbb{X}, \mathbb{Y})$ between compact $\RCD(K, N)$ spaces  $\mathbb{X}, \mathbb{Y}$.

It is trivial that $\dist_{\mathrm{Spec}}^t$ is still symmetric and that the triangle inequality for $\dist_{\mathrm{Spec}}^t$ also holds.
However the nondegeneracy of $\dist_{\mathrm{Spec}}^t$  for compact $\RCD(K, N)$ spaces is \textit{not} satisfied as the following observation shows;
for all $c, t \in (0, \infty)$ and spectral datum $\mathbf{a}$ of $\mathbb{X}$, letting $\mathbb{Y}:=(X, \dist_X, c^{-2} \meas_X)$ and $\mathbf{b}=(c\phi_0^{\mathbf{a}}, c\phi_1^{\mathbf{a}}, \ldots )$, we have 
\begin{equation}\label{eq:amb}
I_{t}^{\mathbf{a}}(x)=I_{t}^{\mathbf{b}}(x), \quad \forall x \in X=Y.
\end{equation}
In particular although $\dist_{\mathrm{Spec}}^t(\mathbb{X}, \mathbb{Y})=0$ holds, but $\mathbb{X}$ is \textit{not} isometric to $\mathbb{Y}$ as metric measured spaces whenever $c \neq 1$. Here let us recall the meaning of an isometry between metric measured spaces; a map $f:X \to Y$ between two metric measured spaces $\mathbb{X}, \mathbb{Y}$ is said to be an \textit{isometry} (then we call that $\mathbb{X}$ is \textit{isometric} to $\mathbb{Y}$) as metric measured spaces if $f$ is bijective, it preserves the distances, and its push-forward measure $f_{\sharp}\meas_X$ by $f$ coincides with $\meas_Y$ (see the beginning of subsection \ref{rcddef}).

We are now in a position to introduce a main result about the nondegeneracy. Roughly speaking the above is only the case when $\dist^t_{\mathrm{Spec}}=0$.
\begin{theorem}\label{thmnond}
Let $\mathbb{Y}=(Y, \dist_Y, \meas_Y)$ be a compact $\RCD(K, N)$ space. If
\begin{equation}\label{subset}
I^{\mathbf{a}}_t(X) \subset I^{\mathbf{b}}_t(Y)
\end{equation}
holds for some $t \in (0, \infty)$ and some spectral data $\mathbf{a}, \mathbf{b}$ of $\mathbb{X}, \mathbb{Y}$, respectively, then $\mathbb{X}$ is isometric to $\mathbb{Y}$ as metric measured spaces, up to a multiplication of $\meas_Y$ by a positive constant.
\end{theorem}
We will also give a quantitative result of the above, see Theorem \ref{from} with Proposition \ref{prop:rela}.
Since it is easy to see that  if $\dist^t_{\mathrm{Spec}}(\mathbb{X}, \mathbb{Y})=0$, then $I^{\mathbf{a}}_t(X)=I^{\mathbf{b}}_t(Y)$ for some $\mathbf{a}, \mathbf{b}$, we obtain the following as a corollary of the theorem above. 
\begin{corollary}[Nondegeneracy]\label{cormain}
If a compact $\RCD(K, N)$ space $\mathbb{Y}=(Y, \dist_Y, \meas_Y)$ satisfies $\dist^t_{\mathrm{Spec}}(\mathbb{X}, \mathbb{Y})=0$ and $\meas_X(X)=\meas_Y(Y)$, then $\mathbb{X}$ is isometric to $\mathbb{Y}$ as metric measured spaces.
\end{corollary}
Let us emphasize here that the results above seem to be new even in the weighted smooth framework. Moreover they show that the framework of  compact $\RCD(K, N)$ spaces gives a positive answer to (Q$1$).

In order to introduce another main result, Theorem \ref{mainthm2},  we denote the eigenvalues of the operator $-\Delta_X$, \textit{without} counting multiplicities, by
\begin{equation}
0=\mu_0(\mathbb{X})<\mu_1(\mathbb{X}) < \mu_2(\mathbb{X}) < \cdots \to \infty
\end{equation}
and then we define $\nu_i(\mathbb{X})$ by the multiplicity of $\mu_i(\mathbb{X})$.
Recall that we say that a sequence of compact $\RCD(K, N)$ spaces $\mathbb{X}_i$ \textit{measured Gromov-Hausdorff} (mGH) \textit{converges} to a compact $\RCD(K, N)$ space $\mathbb{X}$ if there exists a sequence of Borel measurable maps $f_i:X_i \to X$ such that the following three conditions are satisfied;
\begin{enumerate}
\item we have as $i \to \infty$
\begin{equation}
\sup_{x_i, y_i \in X_i}\left| \dist_{X_i}\left(x_i, y_i\right)-\dist_X\left(f_i(x_i), f_i(y_i)\right)\right| \to 0;
\end{equation}
\item we have $X=B_{\epsilon_i}(f_i(X_i))$ for some $\epsilon_i \to 0^+$, where $B_{\epsilon}(A)$ denotes the $\epsilon$-open neighborhood of a subset $A$ of $X$;
\item the push-forward measure $(f_i)_{\sharp}\meas_{X_i}$ of $\meas_{X_i}$ by $f_i$ weakly converge to $\meas_X$ in $X$.
\end{enumerate}
See also Definition \ref{defmgh}.

Let us provide the second main result.
\begin{theorem}\label{mainthm2}
Let $\mathbb{X}_i, \mathbb{X}$ be compact $\RCD(K, N)$ spaces $(i=1,2,\ldots)$. Assume that the diameters $\mathrm{diam}(X_i, \dist_i)$ are uniformly bounded. Then the following four conditions are equivalent.
\begin{enumerate}
\item We have
\begin{equation}\label{101}
\dist_{\mathrm{Spec}}^t(\mathbb{X}_i, \mathbb{X}) \to 0 \quad \text{and} \quad \meas_{X_i}(X_i) \to \meas_X(X)
\end{equation} 
for some $t \in (0, \infty)$.
\item (1) is satisfied for any $t \in (0, \infty)$.
\item We have
\begin{equation}\label{22377}
\mathbb{X}_i \stackrel{\mathrm{mGH}}{\to} \mathbb{X},\quad \mu_j(\mathbb{X}_i) \to \mu_j(\mathbb{X})\,\,\,\text{and}\,\,\,\nu_j(\mathbb{X}_i) \to \nu_j(\mathbb{X}), \quad \forall j,
\end{equation}
where the notaion, $\stackrel{\mathrm{mGH}}{\to}$, denotes the mGH convergence. 
\item We have
\begin{equation}\label{asioraoiha}
\mathbb{X}_i \stackrel{\mathrm{mGH}}{\to} \mathbb{X},\quad \mu_j(\mathbb{X}_i) \to \mu_j(\mathbb{X}), \quad \forall j.
\end{equation}
\end{enumerate}
\end{theorem}
This theorem gives a positive answer to (Q$2$). Namely, the spectral convergence with respect to $\dist_{\mathrm{Spec}}^t$ holds with the convergence of total measures if and only if the mGH convergence holds with the convergence of eigenvalues, without counting multiplicities. This provides us a well-understanding of $\dist_{\mathrm{Spec}}^t$-convergence.  In fact, the example (\ref{flat}) is a mGH convergent sequence with the Riemannian volume measures, but $\nu_1(\mathbb{X}_i)$ does not converge to $\nu_1(\mathbb{X})$, and $\mu_2(\mathbb{X}_i)$ converge to $\mu_1(\mathbb{X})$ ($\neq \mu_2(\mathbb{X})$ by definition). Note that Theorem \ref{mainthm2} can be improved in the case when the sequence is noncollapsed, see Corollary \ref{noncollapsed}.

Finally let us give a positive answer to (Q$3$). Note that the spectral distance $\tilde \dist_{\mathrm{Spec}}$ in the sense of Kasue-Kumura is also well-defined for compact $\RCD(K, N)$ spaces (see Proposition \ref{prosss}).
\begin{theorem}\label{kk} 
Under the same setting as in Theorem \ref{mainthm2}, we see that $\mathbb{X}_i$ $\tilde \dist_{\mathrm{Spec}}$-converge to $\mathbb{X}$ if and only if $\mathbb{X}_i$ mGH-converge to $\mathbb{X}$. 
\end{theorem}

\subsection{Strategy of proof}
Firstly let us introduce an outline of the proof of Theorem \ref{thmnond}. Under the assumption (\ref{subset}), we can find a continuous map $f:X \to Y$ preserving the embeddings $I^{\mathbf{a}}_t, I^{\mathbf{b}}_t$. The first ingredient of the proof is to show that the push-forward measure $f_{\sharp}\meas_X$ is equal to $c\meas_Y$ for some positive constant $c \in (0, \infty)$.

In the case of closed Riemannian manifolds, this is done in \cite{BerardBessonGallot} after proving the smoothness of $f$. Note that we can follow the same arguments as in \cite{BerardBessonGallot} if 
\begin{equation}\label{mm}
f_{\sharp}\meas_X \ll \meas_Y \quad \text{and} \quad \frac{\di f_{\sharp}\meas_X}{\di \meas_Y} \in L^2(Y, \meas_Y)
\end{equation}
hold.  However it may be hard to check (\ref{mm}) directly in this setting (for instance the Lipschitz continuity of $f$ is not enough to get (\ref{mm}) in general).
In order to overcome this difficulty by a simple trick, we use the dual heat flow $\tilde h_s$ acting on the space of all Borel probability measures on $Y$ to prove that $\tilde h_s(f_{\sharp}\meas_X /\meas_X(X))$ is equal to $\meas_Y$ up to a multiplication of a positive constant to $\meas_Y$. Then letting $s \to 0^+$ yields the desired equality $f_{\sharp}\meas_X=c\meas_Y$. In particular since the support of $f_{\sharp}\meas_X$ is $Y$, we easily see that $f$ is surjective.
Following \cite{BerardBessonGallot} with $f_{\sharp}\meas_X=c\meas_Y$, it holds that $f$ also preserves the eigenvalues and the eigenfunctions. In particular $f$ preserves $I^{\mathbf{a}}_s$ and $I_s^{\mathbf{b}}$ for \textit{any} $s \in (0, \infty)$. Therefore letting $s \to 0^+$ with the Varadhan type asymptotics established in \cite{JiangLiZhang} (see (\ref{varad})) allows us to conclude that $f$ preserves the distances. Note that this argument is also different from the one in \cite{BerardBessonGallot} as explained just after (\ref{p}). We thus have Theorem \ref{thmnond}.

Secondly we provide a sketch of the proof of Theorem \ref{mainthm2}. As a first step, we recall that under a mGH-convergence of compact $\RCD(K, N)$ spaces;
\begin{equation}
\mathbb{X}_i \stackrel{\mathrm{mGH}}{\to} \mathbb{X},
\end{equation}
the eigenvalues behave continuously as follows;
\begin{equation}\label{444}
\lambda_j(\mathbb{X}_i)\to \lambda_j(\mathbb{X}),\quad \forall j.
\end{equation}
Moreover the eigenfunctions also behave continuously with respect to the uniform and the $H^{1,2}$-strong convergence. This is proved in \cite{GigliMondinoSavare13} for $\RCD$ spaces (see also \cite{CheegerColding3, Fukaya}), and this behavior is also called a spectral convergence in the previous literatures. 

Then it is not hard to check by (\ref{444}) that the last two conditions in Theorem \ref{mainthm2} are equivalent to each other. Furthermore the remaining implications are done by applying the previous convergence results for $I^{\mathbf{a}}_t$ proved in \cite{AHPT} with Corollary \ref{cormain} (or Theorem \ref{thmnond}).

Finally let us add a comment on the proof of Theorem \ref{kk}.
After proving that $\tilde \dist_{\mathrm{Spec}}$ determines actually a distance even for compact $\RCD(K, N)$ spaces (see Proposition \ref{prosss}), the desired conclusion comes from the uniform convergence of heat kernels with respect to the mGH-convergence proved in \cite{AmbrosioHondaTewodrose, ZZ} and an elementary fact in topology that any bijective continuous map from a compact space to a Hausdorff space is indeed a homeomorphism.

\subsection{Organization of the paper}
In the next section, Section \ref{77}, we give a brief introduction on $\RCD$ spaces, mainly focusing on eigenfunctions/heat kernels for our purposes. Section \ref{7788} is devoted to the proofs of results about the spectral convergence in the sense of B\'erard-Besson-Gallot. As a related observation, we establish a \textit{canonical} mGH-approximation between compact $\RCD(K, N)$ spaces via spectral data, which has an independent interest. See Proposition \ref{prop:rela} and Theorem \ref{from}.
In Section \ref{77889}, we prove Theorem \ref{kk}. The final section, Section \ref{finalfinal}, discusses the case when the limit space is a single point, as the remaining case.

\textbf{Acknowledgement.}
A part of the work is done during the stay at the Fields Institute for the Thematic Program on Nonsmooth Riemannian and Lorentzian Geometry.
The  author wishes to thank the instituite and all of the organizers for their warm hospitalities. Moreover he would like to thank the referee for his/her very careful reading with
valuable suggestions for a revision, which greatly improved the presentation of the paper.
He acknowledges supports of the Grant-in-Aid for
Scientific Research (B) of 20H01799, the Grant-in-Aid for Scientific Research (B) of 21H00977 and Grant-in-Aid for Transformative Research Areas (A) of 22H05105.
\section{$\RCD$ space}\label{77}
In order to keep our short presentations in the paper, we assume that the readers are familiar with the theory of $\RCD$ spaces. The main purpose of this section is to provide a brief introduction on the theory. See \cite{A} for a nice survey on this topic.
\subsection{Definition and terminology}\label{rcddef}
A triple $(X, \dist_X, \meas_X)$, denoted by $\mathbb{X}$ for short, is said to be a \textit{metric measured space} if $(X, \dist_X)$ is a complete separable metric space and $\meas_X$ is a Borel measure on $X$ with full support and $\meas_X(A)<\infty$ for any bounded Borel subset $A$ of $X$. A map between metric measured spaces, $f:X \to Y$, is said to be an \textit{isometry} as metric measured spaces if it is an isometry as metric spaces and the push-forward measure $f_{\sharp}\meas_X$ of $\meas_X$ by $f$ coincides with $\meas_Y$.  

Fix a metric measured space $\mathbb{X}=(X, \dist_X, \meas_X)$. Then the $H^{1,2}$-\textit{Sobolev space}, denoted by $H^{1,2}(\mathbb{X})$, is defined by the finiteness domain of the \textit{Cheeger energy} $\Ch:L^2(X, \meas)\to [0, \infty]$ (see also \cite{Cheeger, Shanmugalingam}). Note that the Cheeger energy $\Ch(f)$ of $f \in H^{1,2}(\mathbb{X})$ can be written in terms of the canonical object $|\nabla f| \in L^2(X, \meas_X)$, called the \textit{minimal relaxed slope} of $f$. We say that $\mathbb{X}$ is \textit{infinitesimally Hilbertian}, written by IH for short below, if the $H^{1,2}$-Sobolev space is a Hilbert space. Note that if $\mathbb{X}$ is IH, then for all $f, g \in H^{1,2}(\mathbb{X})$, the pointwise inner product $\langle \nabla f, \nabla g\rangle (x)$ makes sense for $\meas_X$-a.e. $x \in X$. The domain  of the \textit{Laplacian} $\Delta_X$ of $\mathbb{X}$, denoted by $D(\Delta_X)$,  is defined by the set of all $f \in H^{1,2}(\mathbb{X})$ satisfying that there exists (a unique) $\phi \in L^2(X, \meas_X)$, denoted by $\Delta_X f$, such that
\begin{equation}
\int_X\langle \nabla f, \nabla g\rangle \di \meas_X=-\int_X\phi g\di \meas_X,\quad \forall g \in H^{1,2}(\mathbb{X}).
\end{equation}

We are now in a position to introduce the definition of $\RCD$ spaces. We say that $(X, \dist, \meas)$ is an $\RCD(K, N)$ \textit{space}, or $\RCD$ \textit{space} for short, for some $K \in \mathbb{R}$ and some $N \in [1, \infty]$ if the following four conditions are satisfied;
\begin{enumerate}
\item{(Riemannian assumption)} it is IH; 
\item{(Sobolev-to-Lipschitz property)} if a $H^{1,2}$-function $f \in H^{1,2}(\mathbb{X})$ satisfies $|\nabla f|(x) \le 1$ for $\meas_X$-a.e. $x \in X$, then $f$ has a $1$-Lipschitz representative;
\item{(volume growth condition)} there exist $x \in X$ and $C \in [1, \infty)$ such that $\meas_X(B_r(x)) \le Ce^{Cr^2}$ holds for any $r \in [1, \infty)$, where $B_r(x)$ denotes the open ball of radius $r$ centered at $x$;
\item{(weak Bochner inequality)} the Bochner inequality;
\begin{equation}
\frac{\Delta_X |\nabla f|^2}{2}\ge \frac{(\Delta_X f)^2}{N}+\langle \nabla \Delta_X f, \nabla f\rangle +K|\nabla f|^2
\end{equation}
is satisfied in a weak sense.  
\end{enumerate}
See \cite{AmbrosioGigliSavare14, AmbrosioGigliMondinoRajala, Gigli1} for the precise definition and see \cite{AmbrosioMondinoSavare, CavallettiMilman, ErbarKuwadaSturm} for equivalent definitions we are adopting. 
See also for instance \cite{BrueNaberSemola, BPS, BrueSemola, Gigli, KM, MondinoNaber} for the structure theory and recent developments.

\textit{Unless otherwise stated in the sequel (in particular except for the final section, Section \ref{finalfinal}), whenever we discuss $\RCD(K, N)$ spaces, we assume that they are not single points.} 

Finally in order to simplify our notations below, let us introduce the following.
\begin{definition}[$\mathcal{M}(K, N, d, v)$]\label{defm}
For all $K \in \mathbb{R}, N \in [1, \infty), d \in [1, \infty)$ and $v \in [1, \infty)$, let us denote by $\mathcal{M}=\mathcal{M}(K, N, d, v)$ the set of all compact $\RCD(K, N)$ spaces $\mathbb{X}$, up to isometries, such that $\mathrm{diam}(X, \dist_X) \in [d^{-1}, d]$ and $\meas_X(X) \in [v^{-1}, v]$ are satisfied.
\end{definition}
It is known from, for instance, \cite[Theorem 3.22]{ErbarKuwadaSturm} that $\mathcal{M}(K, N, d, v)$ with the measured Gromov-Hausdorff topology is compact because of the finiteness of $N$, which will be discussed in subsection \ref{mGHsub}.
Note:
\begin{itemize}
\item the finiteness of $N$ is essential to get the compactness of $\mathcal{M}$, namely, for instance, a sequence of unit closed balls $B^n$ in $\mathbb{R}^n$ with the canonical probability measures does not have a mGH convergent subsequence as $n \to \infty$. Thus we always assume that $N$ is finite in the sequel; 
\item though the definition above still makes sense in the case when $N<1$, in this case, we can prove that $X$ is a single point.
\end{itemize}

\subsection{Heat flow and its kernel}\label{subsec:heat}
This subsection is devoted to introducing the main topics of the paper; heat kernels and eigenfunctions. 
Let us fix an $\RCD(K, N)$ space $\mathbb{X}$ for some $K \in \mathbb{R}$ and some $N \in [1, \infty)$ in the sequel. 
Note that we will immediately use standard notations in this topic, for example, $C(a_1, a_3, \ldots, a_k)$ denotes a positive constant depending only on $a_1, a_2, \ldots, a_k$.
\subsubsection{General case}
The \textit{heat flow} of $\mathbb{X}$ starting at $f \in L^2(X, \meas_X)$ is defined by the absolutely continuous (or equivalently, smooth, in this setting, see \cite{GP2}) curve;
\begin{equation}
h_{\cdot}f:(0, \infty) \to L^2(X, \meas_X)
\end{equation}
satisfying that
$h_tf \in D(\Delta)$ holds for any $t \in (0, \infty)$, that $h_tf \to f$ in $L^2(X, \meas_X)$ as $t \to 0^+$ and that
\begin{equation}
\frac{\di}{\di t}h_tf=\Delta_X h_tf, \quad \forall t \in (0, \infty).
\end{equation}
Then the \textit{heat kernel $p_X(x, y, t)$ of $\mathbb{X}$} is determined by the continuous function $p_X:X \times X \times (0, \infty) \to \mathbb{R}$ satisfying that for any $f \in L^2(X, \meas_X)$, we have
\begin{equation}\label{eq:heatcon}
h_tf(x)=\int_Xf(y)p_X(x, y, t)\di \meas_X (y), \quad \text{for $\meas_X$-a.e.}\,\,\,x \in X.
\end{equation}
See also \cite{Sturm95, Sturm96}.
The sharp Gaussian (gradient) estimates on $p_X$  proved in \cite[Theorem 1.2 and Corollary 1.2]{JiangLiZhang} (see also \cite{GarofaloMondino, Jiang15}) are stated as follows.
\begin{theorem}[Gaussian estimate]\label{thm:gaussian}
For any $\epsilon>0$, there exists $C=C(K, N, \epsilon) \in (1, \infty)$ such that
\begin{equation}\label{eq:gaussian}
\frac{C^{-1}}{\meas_X (B_{\sqrt{t}}(x))}\exp \left(-\frac{\dist_X (x, y)^2}{(4-\epsilon)t}-Ct \right) \le p_X(x, y, t) \le \frac{C}{\meas_X (B_{\sqrt{t}}(x))}\exp \left( -\frac{\dist_X (x, y)^2}{(4+\epsilon)t}+Ct \right)
\end{equation}
holds for all $x,\, y \in X$ and $t \in (0, \infty)$ and that
\begin{equation}\label{eq:equi lip}
|\nabla_x p_X(x, y, t)|\le \frac{C}{\sqrt{t}\meas_X (B_{\sqrt{t}}(x))}\exp \left(-\frac{\dist_X(x, y)^2}{(4+\epsilon) t}+Ct\right)
\qquad\text{for $\meas_X$-a.e. $x\in X$}
\end{equation}
holds for all $t\in (0, \infty)$ and $y\in X$. In particular the Varadhan type asymptotics is also satisfied in this setting;
\begin{equation}\label{varad}
-4t\log p_X(x, y, t) \to \dist_X(x, y)^2, \quad \forall x, y \in X,\,\,\,\text{as}\,\,\,t\to0^+.
\end{equation}
\end{theorem}
Note that the heat flow can be extended to a map $h_t:L^p(X, \meas_X) \to L^p(X, \meas_X)$ for any $p \in [1, \infty]$ by satisfying (\ref{eq:heatcon}), and then we have 
\begin{equation}\label{s9asabryasabras}
h_t f \in L^{p}(X, \meas_X)\cap C(X),\quad \forall p \in [1, \infty],\quad \forall f \in L^p(X, \meas_X),\quad \forall t \in (0, \infty)
\end{equation}
because of  (\ref{eq:heatcon}), (\ref{eq:gaussian}) and (\ref{eq:equi lip}). Moreover the heat flow $\tilde{h}_t$ also acts on the set of all Borel probability measures on $X$, denoted by $\mathcal{P}(X)$, as the dual of $h_t$, namely $\tilde{h}_t:\mathcal{P}(X) \to \mathcal{P}(X)$ is defined by satisfying
\begin{equation}
\int_Xf\di \tilde{h}_t \mathfrak{n}=\int_Xh_tf\di \mathfrak{n}
\end{equation}
for any continuous functions $f$ on $X$ with bounded supports.
Then it is proved in \cite[Theorem 4.8]{AmbrosioGigliSavare15} (or just by applying (\ref{eq:equi lip})) that $\tilde{h}_t\mathfrak{n}=n_t\meas_X$ holds for some $n_t \in L^1(X, \meas_X)$ with a log-Harnack inequality  (see also \cite[Proposition 4.4]{AmbrosioHonda}). Since $\tilde{h}_t\mathfrak{n}=n_t\meas_X=(h_{t/2}n_{t/2})\meas_X$ holds, we know $n_t \in L^{1}(X, \meas_X)\cap C(X)$
because of applying (\ref{s9asabryasabras}) as $p=1$.
 This observation will  play a role in the proof of Lemma \ref{lem:orth} below.
\subsubsection{Compact case}\label{spectraltheory}
From now on we assume that $(X, \dist_X)$ is compact. 
For any $\lambda \in \mathbb{R}$, denoting by $E_{\lambda}(\mathbb{X})=\{f \in D(\Delta_X); \Delta_X f+\lambda f=0\}$, the \textit{multiplicity} of $\lambda$, denoted by $\nu_{\mathbb{X}}(\lambda)$,  is defined by the dimension of the linear space $E_{\lambda}(\mathbb{X})$.
Then we say that a real number $\lambda$ is an \textit{eigenvalue} of $-\Delta_X$ if $\nu_{\mathbb{X}}(\lambda) \ge 1$ holds (then $\lambda \ge 0$ holds). Any nonzero element of $E_{\lambda}(\mathbb{X})$ for an eigenvalue $\lambda$ is called an \textit{eigenfunction}.

On the other hand, the Bishop-Gromov inequality (see \cite[Theorem 5.31]{LottVillani} and \cite[Theorem 2.3]{Sturm06b}) and a Poincar\'e inequality (see \cite[Theorem 1]{Rajala}) show that the canonical inclusion $H^{1, 2}(\mathbb{X}) \hookrightarrow L^2(X, \meas_X)$ is a compact operator (see \cite[Theorem.8.1]{HK}).
This observation with a standard technique in functional analysis allows us to denote by 
\begin{equation}\label{eq:eign}
0=\lambda_0(\mathbb{X})<\lambda_1(\mathbb{X}) \le \lambda_2(\mathbb{X}) \le \cdots \to \infty
\end{equation}
the eigenvalues of the operator $-\Delta_X$ counted with multiplicities (see for instance \cite[Appendix]{Honda3}). On the other hand, \textit{without} counting multiplicities, we denote the eigenvalues by
\begin{equation}
0=\mu_0(\mathbb{X})<\mu_1(\mathbb{X}) < \mu_2(\mathbb{X}) < \cdots \to \infty
\end{equation}
and then we define $\nu_i(\mathbb{X})$ to be the multiplicity of $\mu_i(\mathbb{X})$, namely $\nu_i(\mathbb{X}):=\nu_{\mathbb{X}}(\mu_i(\mathbb{X}))$. Thus by definition, (\ref{eq:eign}) can be written by
\begin{equation}
0=\mu_0(\mathbb{X})<\underbrace{\mu_1(\mathbb{X}) = \cdots =\mu_1(\mathbb{X})}_{\nu_1(\mathbb{X})} <\underbrace{\mu_2(\mathbb{X})= \cdots = \mu_2(\mathbb{X})}_{\nu_2(\mathbb{X})}< \mu_3(\mathbb{X}) = \cdots \to \infty.
\end{equation}
Let us recall that for a sequence $\phi_i^X$ of corresponding eigenfunctions for the eigenvalues $\lambda_i(\mathbb{X})$ with $\|\phi_i^X\|_{L^2}=1$, $\{\phi_i^X\}_i$ is an $L^2$-orthonormal basis of $L^2(X, \meas_X)$.
 Moreover it is known that 
\begin{equation}\label{eq:expansion1}
p_X(x,y,t) = \sum_{i \ge 0} e^{- \lambda_i(\mathbb{X}) t} \phi_i^X(x) \phi_i^X (y),\quad \forall x, \,\,\forall y \in X, \,\,\forall t \in (0, \infty)
\end{equation}
holds and that
\begin{equation}\label{eq:expansion2}
p_X(\cdot,y,t) = \sum_{i \ge 0} e^{- \lambda_i(\mathbb{X}) t} \phi_i^X(y) \phi_i^X \qquad \text{in $H^{1,2}(\mathbb{X})$}
\end{equation}
holds for all $y\in X$ and $t \in (0, \infty)$. 
Combining (\ref{eq:expansion1}) and (\ref{eq:expansion2}) with (\ref{eq:equi lip}), we know the following quantitative estimates on eigenvalues/eigenfunctions;
\begin{equation}\label{eq:eigenfunction}
\|\phi_i^X\|_{L^\infty} \leq \tilde{C} \lambda_i(\mathbb{X})^{N/4}, \qquad \| \nabla \phi_i^X \|_{L^\infty} \leq \tilde{C} \lambda_i(\mathbb{X})^{(N+2)/4}, \qquad \lambda_i(\mathbb{X}) \ge \tilde{C}^{-1}i^{2/N}
\end{equation}
if $\mathrm{diam}(X, \dist_X) \in (0, d]$ holds for some positive number $d$,
where $\tilde{C}:=\tilde{C}(K, N, d) \in (1, \infty)$ (see the appendix of \cite{AHPT} and see also \cite{Jiang}). In particular (\ref{eq:eigenfunction}) implies that (\ref{eq:expansion1}) is also satisfied in $C(X\times X)$ for any fixed $t \in (0, \infty)$.

Finally let us end this subsection by giving the following technical lemma which will play a role in the proof of Theorem \ref{thmdeg}.
\begin{lemma}\label{lem:orth}
If a Borel measure $\mathfrak{n}$ on $X$ satisfies $\mathfrak{n} (X) \in (0, \infty)$ and
\begin{equation}
\int_X\phi^X_i\di \mathfrak{n}=0,\quad \forall i \ge 1,
\end{equation}
then $\mathfrak{n} =a \meas$ for some positive constant $a \in (0, \infty)$.
\end{lemma}
\begin{proof}
Without loss of generality we can assume $\mathfrak{n}(X)=1$.
For any $t>0$, denoting by $\tilde{h}_t\mathfrak{n}=n_t\meas$ for some $n_t \in L^{1}(X, \meas_X) \cap C(X)$ (thus $n_t \in L^2(X, \meas_X)$ because of the compactness of $X$), we have for any $i \ge 1$
\begin{equation}
\int_X\phi^X_i n_t\di \meas=\int_X\phi^X_i\di \left( \tilde{h}_t\mathfrak{n}\right) =\int_Xh_t\phi^X_i \di \mathfrak{n} = e^{-\lambda_i(\mathbb{X}) t}\int_X\phi_i^X \di \mathfrak{n}=0.
\end{equation}
Thus $n_t$ is a constant. Then letting $t \to 0^+$ completes the proof. 
\end{proof}
\section{Spectral distance in the sense of B\'erard-Besson-Gallot}\label{7788}
In this section we discuss embeddings $I^{\mathbf{a}}_t$ of compact $\RCD(K, N)$ spaces into $\ell^2$ as explained in the introduction. 
\subsection{Embedding via eigenfunctions}\label{88shsnsn}
Let us fix a compact $\RCD(K, N)$ space $\mathbb{X}$ for some $K \in \mathbb{R}$ and some $N \in [1, \infty)$, and fix $t \in (0, \infty)$. Then it is proved in \cite[subsection 4.2]{AHPT} that the map $\Phi_t^X:X \hookrightarrow L^2(X, \meas_X)$ defined by $\Phi_t^X(x):=(y\mapsto p_X(x,y,t))$ is a topological embedding. Thanks to (\ref{eq:equi lip}), we know that
$\Phi_t^X$ is $C(K, N, d, v, t_0)$-Lipschitz if $\mathbb{X} \in \mathcal{M}(K, N, d, v)$ and  $t \in [t_0^{-1}, t_0]$ are satisfied for some positive constants $d, v, t_0 \in [1, \infty)$.
In order to replace the target space $L^2(X, \meas_X)$ of $\Phi_t^X$ by $\ell^2$, we introduce the following.
\begin{definition}[Spectral data]
A sequence of functions on $X$
\begin{equation}
\mathbf{a} =(\phi_0^{\mathbf{a}}, \phi_1^{\mathbf{a}}, \phi_2^{\mathbf{a}},\ldots)
\end{equation}
is said to be a \textit{spectral datum} of $\mathbb{X}$ if $\Delta_X \phi_i^{\mathbf{a}}+\lambda_i(\mathbb{X})\phi_i^{\mathbf{a}}=0$ holds with $(\phi_i^{\mathbf{a}}, \phi_j^{\mathbf{a}})_{L^2}=\delta_{ij}$ (thus $\{\phi_i^{\mathbf{a}}\}_i$ is an orthonormal basis of $L^2(X, \meas_X)$). 
\end{definition}
Fix a spectral datum $\mathbf{a}$ of $\mathbb{X}$. Denoting by $\iota^{\mathbf{a}}:L^2(X, \meas_X) \simeq \ell^2$ the canonical isometry via $\mathbf{a}$, 
an embedding $\Phi_t^{\mathbf{a}}:X \hookrightarrow \ell^2$ defined by $\Phi_t^{\mathbf{a}}:=\iota^{\mathbf{a}}\circ \Phi_t^X$ is explicitly written by
\begin{equation}
\Phi_t^{\mathbf{a}}(x)=\left(e^{-\lambda_i(\mathbb{X})t}\phi_i^{\mathbf{a}}(x)\right)_{i \ge 0}
\end{equation}
because of (\ref{eq:expansion1}). Following \cite{BerardBessonGallot}, let us discuss a normalized embedding;
\begin{equation}
I^{\mathbf{a}}_t:X \hookrightarrow \ell^2
\end{equation}
defined by
\begin{equation}
I^{\mathbf{a}}_{t}(x):=\left(\sqrt{\meas_X (X)}e^{-\lambda_i(\mathbb{X})t/2}\phi_i^{\mathbf{a}}(x)\right)_{ i \ge 1}.
\end{equation}
\begin{remark}\label{scalei}
As a generalization of (\ref{eq:amb}), for all $c, s \in (0, \infty)$, letting $\mathbb{Z}:=(X, \sqrt{t/s}\dist_X, c^{-2}\meas_X)$ and $\mathbf{c}:=(c\phi_0^{\mathbf{a}}, c\phi_1^{\mathbf{a}},\ldots )$, we have
\begin{equation}
I^{\mathbf{a}}_t(x)=I^{\mathbf{c}}_s(x),\quad \forall x \in X=Z.
\end{equation}
This observation tells us that when we wish to compare $I^{\mathbf{a}}_t$ with $I^{\mathbf{b}}_s$ for different $s, t$, after rescaling the metrics, the situation is reduced to the case when $s=t$.
\end{remark}
Let us prove that any embedding $I^{\mathbf{a}}_t$ reconstructs the metric measured structure of $\mathbb{X}$ in the following sense. 
\begin{theorem}\label{thmdeg}
Let $\mathbb{Y}$ be a compact $\RCD(K,N)$ space and let $\mathbf{b}$ be a spectral datum of $\mathbb{Y}$. If a map $f:X \to Y$ satisfies
\begin{equation}\label{condi}
I^{\mathbf{a}}_{t}(x)=I^{\mathbf{b}}_{t}(f(x)),\quad \forall x \in X,
\end{equation}
namely \begin{equation}\label{eq:comp3}
\sqrt{\meas_X(X)}e^{-\lambda_i(\mathbb{X})t/2}\phi_i^{\mathbf{a}}(x)=\sqrt{\meas_Y(Y)}e^{-\lambda_i(\mathbb{Y})t/2}\phi_i^{\mathbf{b}}(f(x)),\quad \forall i \ge 1,\quad \forall x \in X,
\end{equation}
then $f$ gives an isometry as metric measured spaces, up to a multiplication of $\meas_Y$ by a positive constant.
\end{theorem}
\begin{proof}
Firstly let us prove that $f$ is continuous (then in particular $f$ is Borel measurable).
Take a convergent sequence $x_j \to x$ in $X$. After passing to a subsequence, with no loss of generality, we can assume that $f(x_j) \to y$ in $Y$ for some $y \in Y$. Then since for any $i \ge 1$
\begin{align}
\sqrt{\meas_Y(Y)}e^{-\lambda_i(\mathbb{Y})t/2}\phi_i^{\mathbf{b}}(y)&=\lim_{j \to \infty}\sqrt{\meas_Y(Y)}e^{-\lambda_i(\mathbb{Y})t/2}\phi_i^{\mathbf{b}}(f(x_j)) \nonumber \\
&=\lim_{j \to \infty}\sqrt{\meas_X(X)}e^{-\lambda_i(\mathbb{X})t/2}\phi_i^{\mathbf{a}}(x_j) \nonumber \\
&=\sqrt{\meas_X(X)}e^{-\lambda_i(\mathbb{X})t/2}\phi_i^{\mathbf{a}}(x) =\sqrt{\meas_Y(Y)}e^{-\lambda_i(\mathbb{Y})t/2}\phi_i^{\mathbf{b}}(f(x)),
\end{align}
we have $\phi_i^{\mathbf{b}}(y)=\phi_i^{\mathbf{b}}(f(x))$.
In particular, thanks to (\ref{eq:expansion1}), it holds that
\begin{equation}\label{eq:varr}
p_{Y}(y, w, s)=p_{Y}(f(x), w, s),\quad \forall w \in Y,\,\,\,\forall s \in (0, \infty).
\end{equation}
Thus letting $s \to 0^+$  in (\ref{eq:varr})  with (\ref{varad}) shows $\dist_Y(y, w)=\dist_Y(f(x), w)$ for any $w \in Y$, thus $y=f(x)$, which proves $f(x_j) \to f(x)$. Thus $f$ is continuous.

Next we consider a Borel measure $f_{\sharp}\meas_X$ on $Y$. 
Since  $f_{\sharp}\meas_X(Y)=\meas_X(X) \in (0, \infty)$ and
\begin{align}
\int_Y\phi_i^{\mathbf{b}}\di (f_{\sharp}\meas_X)&=\int_X\phi^{\mathbf{b}}_i(f(x))\di \meas_X(x) \nonumber \\
&= \left( \frac{\meas_X(X)}{\meas_Y(Y)}\right)^{1/2}\cdot e^{-(\lambda_i(\mathbb{X})-\lambda_i(\mathbb{Y}))t/2} \cdot \int_X\phi_i^{\mathbf{a}}\di \meas_X \nonumber \\
&=0,\quad \forall i \ge 1,
\end{align}
Lemma \ref{lem:orth} allows us to conclude that $f_{\sharp}\meas_X =  c \meas_Y$ holds for some positive constant $c \in (0, \infty)$ (thus $c=\meas_X(X)/\meas_Y(Y)$).

In particular the support of $f_{\sharp}\meas_X$ coincides with $Y$.
Therefore $f$ is surjective because if $Y \setminus f(X) \neq \emptyset$, then since $f(X)$ is compact, we can find a ball $B$ included in $Y \setminus f(X)$, thus $f_{\sharp}\meas_X(B)=\meas_X(\emptyset)=0$, which contradicts a fact that $\supp f_{\sharp}\meas_X=Y$.

Note that (\ref{eq:comp3}) yields
\begin{equation}\label{eq:comp4}
\meas_X(X)e^{-\lambda_i(\mathbb{X})t}(\phi_i^{\mathbf{a}}(x))^2=\meas_Y(Y)e^{-\lambda_i(\mathbb{Y})t}(\phi_i^{\mathbf{b}}(f(x)))^2,\quad \forall i \ge 1.
\end{equation}
Integrating this over $X$ shows
\begin{equation}
\meas_X(X)e^{-\lambda_i(\mathbb{X})t}=\meas_Y(Y)e^{-\lambda_i(\mathbb{Y})t} c \int_Y(\phi_i^{\mathbf{b}})^2\di \meas_Y =\meas_X(X)e^{-\lambda_i(\mathbb{Y})t}.
\end{equation}
Thus $\lambda_i(\mathbb{X})=\lambda_i(\mathbb{Y})$ for any $i$. In particular combining this with (\ref{eq:comp3}) yields
\begin{equation}\label{eq:heatasa}
\meas_X(X) p_{X}(x, w, s)=\meas_Y(Y)p_{Y}(f(x), f(w), s),\quad \forall x, w \in X,\,\,\,\forall s \in (0, \infty).
\end{equation}
Thus (\ref{varad}) under letting $s\to 0^+$ in  (\ref{eq:heatasa}) shows that $f$ preserves the distance, namely $f$ gives an isometry as metric spaces. Therefore we conclude.
\end{proof}
We are now in a position to prove Theorem \ref{thmnond}.
\begin{proof}[Proof of Theorem \ref{thmnond}]
Let us define a map $f:X \to Y$ by $f(x):=(I^{\mathbf{b}}_t)^{-1} \circ I^{\mathbf{a}}_t(x)$, where this is well-defined because of (\ref{subset}). Then since it is trivial that (\ref{condi}) is satisfied, Theorem \ref{thmdeg} completes the proof.
\end{proof}
\subsection{mGH and convergence of spectral data}\label{mGHsub}
Let us fix a measured Gromov-Hausdorff (mGH) convergent sequence of compact $\RCD(K,N)$ spaces for some $K \in \mathbb{R}$ and some $N \in [1, \infty)$;
\begin{equation}\label{eqmgh}
\mathbb{X}_i \xrightarrow[f_i]{\mathrm{mGH}} \mathbb{X}, \quad (\text{or}\,\,\,\mathbb{X}_i \stackrel{\mathrm{mGH}}{\to} \mathbb{X}\,\,\,\text{for short}), \quad \text{in}\,\,\,\mathcal{M}(K, N, d, v),
\end{equation}
where $\{f_i\}_i$ denote fixed $\epsilon_i$-mGH approximations from $X_i$ to $X$ (as explained below) realizing the mGH convergence from $\mathbb{X}_i$ to $\mathbb{X}$ (see Definition \ref{defm} for the definition of $\mathcal{M}(K, N, d, v)$). 

The following is a definition of a mGH approximation which will be compared with the spectral one later. Let us emphasize that in the sequel except for the final subsection, subsection \ref{result}, any sequence of compact $\RCD(K, N)$ spaces we will mainly discuss is always assumed that their diameters and total measures are bounded away from $0$ and $\infty$, namely they are in $\mathcal{M}(K, N, d, v)$ for some $d \in [1, \infty)$ and some $v\in [1, \infty)$.
\begin{definition}[GH/mGH approximation]\label{defmgh}
For two compact $\RCD(K, N)$ spaces $\mathbb{Y}, \mathbb{Z}$, a Borel map $f:Y \to Z$ is said to be an \textit{$\epsilon$-Gromov-Hausdorff (GH) approximation} (as metric spaces) if $|\dist_Z(f(x), f(y))-\dist_Y(x,y)|\le \epsilon$ holds for all $x,y \in Y$ with $Z=B_{\epsilon}(f(Y))$. Furthermore we say that it is an \textit{$\epsilon$-measured Gromov-Hausdorff (mGH) approximation} (as metric measure spaces) if it is an $\epsilon$-GH approximation with
\begin{equation}
\mathsf{W}_Z^2\left(\frac{f_{\sharp}\meas_Y}{\meas_Y(Y)}, \frac{\meas_Z}{\meas_Z(Z)}\right) +\left|\meas_Y(Y)-\meas_Z(Z)\right| \le \epsilon,
\end{equation}
where $\mathsf{W}^2_Z$ denotes the $L^2$-Wasserstein distance on $\mathcal{P}(Z)$ ($=\mathcal{P}_2(Z)$ because of the compactness of $Z$). 
\end{definition}
Recall the \textit{$L^2$-Wasserstein dinstance on $\mathcal{P}(Z)$; for all $\mathfrak{n}_i \in \mathcal{P}(Z)(i=1,2)$, 
\begin{equation}
\mathsf{W}_Z^2\left( \mathfrak{n}_1,  \mathfrak{n}_2\right):=\left( \inf_{\gamma}\int_{Z \times Z}\dist_Z(x,y)^2\di \gamma \right)^{1/2},
\end{equation}
where $\gamma$ runs over all Borel probability measures on $Z \times Z$ with $(\pi_i)_{\sharp}\gamma=\mathfrak{n}_i$, and $\pi_i:Z \times Z \to Z$ denotes the projection of the $i$-th factor.
}

Note that GH/mGH convergences are metrizable by distances $\dist_{\mathrm{GH}}$/$\dist_{\mathrm{mGH}}$, respectively and that their convergences can be also expressed by existences of approximations, as expressed in (\ref{eqmgh}) (see for instance \cite{Sturm06}).

In the setting (\ref{eqmgh}), we say that a sequence of points $x_i \in X_i$ \textit{converges} to a point $x \in X$ if $f_i(x_i) \to x$ in $X$. Note that this convergence notion depends on the choice of the approximations $\{f_i\}_i$. 
\begin{remark}\label{remrem}
This is an intrinsic approach for the mGH convergence. See \cite{GigliMondinoSavare13} for an extrinsic one and the equivalence to our setting.  See also \cite{AmbrosioHonda, AmbrosioHonda2} for related convergence notions for functions, which will be used immendiately in the sequel.
\end{remark}
Let us define the convergence of spectral datas with respect to the mGH convergence. 
\begin{definition}[Convergence of spectral data]
Recall our assumption (\ref{eqmgh}).
We say that a sequence of spectral data $\mathbf{a}_i$ of $\mathbb{X}_i$ \textit{converge} to a spectral datum $\mathbf{a}$ of $\mathbb{X}$ with respect to (\ref{eqmgh}) if $\phi_j^{\mathbf{a}_i}$ pointwisely converge to $\phi_j^{\mathbf{a}}$ for any $j$, namely 
\begin{equation}
\phi_j^{\mathbf{a}_i}(x_i) \to \phi_j^{\mathbf{a}}(x),\quad \text{whenever}\,\,\,x_i \to x.
\end{equation}
Then we denote by $\mathbf{a}_i \to \mathbf{a}$ the convergence.
\end{definition}
Note that the pointwise convergence of $\phi_j^{\mathbf{a}_i}$ implies the uniform and the $H^{1.2}$-strong converences of $\phi_j^{\mathbf{a}_i}$ because of (\ref{eq:eigenfunction}).
The following is a direct consequence of \cite[Theorem 7.8]{GigliMondinoSavare13}.
\begin{theorem}[Compactness of spectral datas]\label{prop:comp}
Recall our assumption (\ref{eqmgh}). Then
any sequence of spectral data $\mathbf{a}_i$ of $\mathbb{X}_i$ has a convergent subsequence to a spectral datum $\mathbf{a}$ of $\mathbb{X}$ with respect to (\ref{eqmgh}). In particular we have
\begin{equation}\label{00}
\lambda_j(\mathbb{X}_i) \to \lambda_j(\mathbb{X}),\quad \forall j.
\end{equation}
\end{theorem}
\begin{corollary}\label{cordef}
The following two conditions are equivalent (recall our assumption (\ref{eqmgh})).
\begin{enumerate}
\item We have
\begin{equation}\label{223}
\mu_j(\mathbb{X}_i) \to \mu_j(\mathbb{X}),\quad \forall j
\end{equation}
and 
\begin{equation}\label{label}
\nu_j(\mathbb{X}_i) \to \nu_j(\mathbb{X}), \quad \forall j.
\end{equation}
\item We have
\begin{equation}\label{jjn}
\mu_j(\mathbb{X}_i) \to \mu_j(\mathbb{X}), \quad \forall j.
\end{equation}
\end{enumerate}
\end{corollary}
\begin{proof}
Let us give a proof of the nontrivial implication, from (2) to (1). Assume that (2) is satisfied.  It is trivial that (\ref{label}) holds for $j=0$ because of the $(2,2)$-Poincar\'e inequality.
The validity of (\ref{jjn}) for $j=2$ implies that there exists $\tau \in (0, 1)$ such that $\mu_2(\mathbb{X}_i) \ge \mu_1(\mathbb{X}_i)+\tau$  holds for any $i$. Combining this with the validity of (\ref{jjn}) for $j=1$ yields that (\ref{label}) holds for $j=1$. Similarly by induction for $j$ we have (\ref{label}) for any $j$, namely (1) holds.
%
\end{proof}
Finally we end this subsection by introducing convergence results on $I_t^{\mathbf{a}}$ with respect to the mGH convergence, proved in \cite[Theorem 5.19]{AHPT}.
\begin{theorem}[Convergence of $I^{\mathbf{a}}_t$]\label{thm:convi}
Recall our assumption (\ref{eqmgh}).
We then have 
\begin{equation}
I_{t_i}^{\mathbf{a}_i}(X_i) \to I_t^{\mathbf{a}}(X)
\end{equation}
with respect to the Hausdorff distance $\dist_{\ell^2}^{\mathrm{H}}$ in $\ell^2$, and 
\begin{equation}
I_{t_i}^{\mathbf{a}_i} \to I_t^{\mathbf{a}}
\end{equation}
with respect to the uniform convergence,
for all convergent sequences of positive numbers $t_i \to t$ in $(0, \infty)$ and of spectral data $\mathbf{a}_i \to \mathbf{a}$. 
\end{theorem}
\subsection{Spectral approximation}
The purpose of this subsection is to construct mGH approximations via spectral embeddings $I^{\mathbf{a}}_t$. 

Let us fix two compact $\RCD(K, N)$ spaces $\mathbb{X}, \mathbb{Y}$ for some $K \in \mathbb{R}$ and some $N \in [1, \infty)$.  Fix $\epsilon, t\in (0, \infty)$ and spectral data $\mathbf{a}, \mathbf{b}$ of $\mathbb{X}, \mathbb{Y}$, respectively. 
Firstly we define a spectral approximation which is a spectral analogue of Definition \ref{defmgh}.
\begin{definition}[Spectral approximation]
We say that a Borel measurable map $f:X \to Y$ is an \textit{$(\epsilon ; \mathbf{a}, \mathbf{b}, t)$-spectral weak approximation} if 
\begin{equation}
\dist_{\ell^2}\left( I^{\mathbf{a}}_{t}(x), I^{\mathbf{b}}_{t}(f(x))\right)<\epsilon,\quad \forall x \in X
\end{equation}
holds. Moreover $f$ is an \textit{$(\epsilon ; \mathbf{a}, \mathbf{b}, t)$-spectral approximation} if it is an $(\epsilon ; \mathbf{a}, \mathbf{b}, t)$-spectral weak approximation and it holds that for any $y \in Y$, there exists $x \in X$ such that 
\begin{equation}\label{eq:suj}
\dist_{\ell^2}\left( I^{\mathbf{b}}_{t}(f(x)), I^{\mathbf{b}}_{t}(y)\right)<\epsilon.
\end{equation}
holds.
\end{definition}
Let us give fundamental properties on spectral (weak) approximations, including the existence. In order to simplify our notations, we use a standard notation in this topic;  $\Psi= \Psi(\cdot; K, N, d, v, t_0)$ denotes a positive function on $(0, \infty)$ depending only on $K, N, d, v$ and $t_0$ such that
\begin{equation}
\lim_{\epsilon\to 0} \Psi(\epsilon; K, N, d, v, t_0)=0
\end{equation} 
holds.
\begin{proposition}\label{prop:rela}
We have the following.
\begin{enumerate}
\item Let $\mathbb{Z}$ be a compact $\RCD(K, N)$ space. If $f:X \to Y$ is an $(\epsilon; \mathbf{a}, \mathbf{b}, t)$-spectral (or weak, respectively) approximation and $g:Y \to Z$ is a $(\delta; \mathbf{b}, \mathbf{c})$-spectral (or weak, respectively) approximation for some spectral datum $\mathbf{c}$ of $\mathbb{Z}$, then $g \circ f$ is an $(\epsilon+\delta; \mathbf{a}, \mathbf{c})$-spectral (or weak, respectively) approximation.
\item If there exists an $(\epsilon; \mathbf{a}, \mathbf{b}, t)$-spectral (or weak, respectively)  approximation from $X$ to $Y$, then 
\begin{equation}\label{submap}
\dist_{\ell^2}^{\mathrm{H}}\left(I^{\mathbf{a}}_t(X), I_t^{\mathbf{b}}(Y)\right)< 5\epsilon \quad \left(\text{or}\,\,\,I^{\mathbf{a}}_t(X) \subset B_{5\epsilon}\left(I_t^{\mathbf{b}}(Y)\right),\,\,\,\text{respectively}\right)
\end{equation}
is satisfied, where $B_{\epsilon}$ denotes the $\epsilon$-open neighborhood.
\item If 
\begin{equation}
\dist_{\ell^2}^{\mathrm{H}}\left(I^{\mathbf{a}}_t(X), I_t^{\mathbf{b}}(Y)\right)< \epsilon \quad \left(\text{or}\,\,\,I^{\mathbf{a}}_t(X) \subset B_{\epsilon}\left(I_t^{\mathbf{b}}(Y)\right),\,\,\,\text{respectively}\right)
\end{equation}
holds, then there exists a $(5\epsilon; \mathbf{a}, \mathbf{b}, t)$-spectral (weak, respectively) approximation from $X$ to $Y$.
\item{(From mGH to spectral approximation)} If a map $f:X \to Y$ is an $\epsilon$-mGH approximation, then 
 $f$ is a $(\Psi;\mathbf{a}, \mathbf{b}, t)$-spectral approximation, whenever $\mathbb{X}, \mathbb{Y} \in \mathcal{M}(K, N, d, v)$ and  $t \in [t_0^{-1}, t_0]$ are satisfied for some positive constants $d, v, t_0 \in [1, \infty)$, where $\Psi=\Psi(\epsilon; K, N, d, v, t_0)$.
\end{enumerate}
\end{proposition}
\begin{proof}
Since (1), (2) and (3) are direct consequences of their definitions, let us give only a proof of (4) by contradiction.
Assume that the assertion is not satisfied. Then there exist a positive number $\tau \in (0, 1]$ and sequences of;
\begin{enumerate}
\item  compact $\RCD(K, N)$ spaces $\mathbb{X}_i$ and $\mathbb{Y}_i$ whose diameters and total measures are bounded away from $0$ and $\infty$;
\item positive numbers $\epsilon_i \to 0^+$;
\item positive numbers $t_i \in (0, \infty)$ which are bounded away from $0$ and $\infty$;
\item spectral data $\mathbf{a}_i, \mathbf{b}_i$ of $\mathbb{X}_i, \mathbb{Y}_i$, respectively, and
\item $\epsilon_i$-mGH approximations $f_i:X_i \to Y_i$ 
\end{enumerate}
such that $f_i$ is not a $(\tau;\mathbf{a}_i, \mathbf{b}_i, t_i)$-spectral approximation. With no loss of generality we can assume that $t_i \to t$ for some $t \in (0, \infty)$. 
Theorem \ref{prop:comp} shows that after passing to a subsequence, we can assume that 
\begin{equation}\label{1000}
\mathbb{X}_i \xrightarrow[g_i]{\mathrm{mGH}} \mathbb{X},\quad \mathbb{Y}_i\xrightarrow[h_i]{\mathrm{mGH}} \mathbb{Y}
\end{equation} 
hold for some compact $\RCD(K, N)$ spaces $\mathbb{X}, \mathbb{Y}$ and that $\mathbf{a}_i, \mathbf{b}_i\to \mathbf{a}, \mathbf{b}$ for some spectral data $\mathbf{a}, \mathbf{b}$ of $\mathbb{X}, \mathbb{Y}$, respectively, with respect to (\ref{1000}).
By the very definition of the mGH approximation, Definition \ref{defmgh}, with no loss of generality we can also assume that $f_i$ converges uniformly to an isometry $f:X \to Y$ as metric measured spaces with respect to (\ref{1000}), in the sense
\begin{equation}\label{snsyshsvsgs}
\sup_{x_i \in X_i}\dist_{Y}\left(h_i\circ f_i(x_i), f\circ g_i(x_i)\right) \to 0,\quad \text{as $i \to \infty$.}
\end{equation}
Since $f$ satisfies (\ref{condi}), Theorem \ref{thm:convi} allows us to conclude that $f_i$ is a $(\delta_i;\mathbf{a}_i, \mathbf{b}_i, t_i)$-spectral approximation for some $\delta_i \to 0^+$ because of the uniform convergence of $f_i$ to $f$. This is a contradiction.
\end{proof}
The next theorem gives a quantitative version of Theorem \ref{thmdeg}. This together with Proposition \ref{prop:rela} proves the existence and the uniqueness of a canonical approximation for given spectral data stated in the abstract.
\begin{theorem}\label{from}
Assume that $\mathbb{X}, \mathbb{Y} \in \mathcal{M}(K, N, d, v)$ and $t \in [t_0^{-1}, t_0]$ are satisfied for some positive constants $d, v, t_0 \in [1, \infty)$.
Then we have the following.
\begin{enumerate}
\item{(From spectral weak approximation to mGH)} Let $f:X \to Y$ be an $(\epsilon; \mathbf{a}, \mathbf{b}, t)$-spectral weak approximation. If $|\meas_X(X)-\meas_Y(Y)|\le \epsilon$ holds, then $f$ is a $\Psi(\epsilon; K, N, d, v, t_0)$-mGH approximation.  
\item{(Uniqueness)} Let $f, g:X \to Y$ be $(\epsilon;\mathbf{a}, \mathbf{b}, t)$-spectral weak approximations. Then 
\begin{equation}
\sup_{x \in X}\dist_Y\left(f(x), g(x)\right) \le \Psi(\epsilon; K, N, d, v, t_0).
\end{equation}
\end{enumerate}
\end{theorem}
\begin{proof}
Firstly let us prove (1) by contradiction.
If the assertion is not satisfied, then as done in the first part of the proof of (4) of Proposition \ref{prop:rela}, there exist a positive number $\tau \in (0,1]$ and sequences of; 
\begin{enumerate}
\item compact $\RCD(K, N)$ spaces
\begin{equation}\label{100}
\mathbb{X}_i \xrightarrow[g_i]{\mathrm{mGH}} \mathbb{X},\quad \mathbb{Y}_i\xrightarrow[h_i]{\mathrm{mGH}} \mathbb{Y}
\end{equation} 
with $|\meas_{X_i}(X_i)-\meas_{Y_i}(Y_i)| \to 0$;
\item positive numbers $\epsilon_i \to 0^+$;
\item positive numbers $t_i \to t$ in $(0, \infty)$;
\item spectral data $\mathbf{a}_i, \mathbf{b}_i\to \mathbf{a}, \mathbf{b}$, respectively with respect to (\ref{100}), and 
\item $(\epsilon_i;\mathbf{a}_i, \mathbf{b}_i, t_i)$-spectral weak approximations $f_i:X_i \to Y_i$
\end{enumerate}
such that each $f_i$ is not a $\tau$-mGH-approximation. 

Let us define a map $f:X \to Y$ as follows;
for any $x \in X$, take a sequence $x_i \in X_i$ converging to $x$ with respect to (\ref{100}), and find a sequence $y_i \in Y_i$ with
\begin{equation}\label{hsihasha8ssn}
\dist_{\ell^2}\left( I^{\mathbf{a}_i}_{t_i}(x_i), I^{\mathbf{b}_i}_{t_i}(y_i)\right) \to 0,
\end{equation}
where such $y_i$ can be found by, for instance, $y_i:=f_i(x_i)$.
Moreover find a sequence $y(i) \in Y$ with $\dist_Y(h_i(y_i), y(i)) \to 0$. Then 
\begin{equation}
f(x):=\lim_{i\to \infty}y(i).
\end{equation}
In order to check that this is well-defined, let us take $\tilde x_i, \tilde y_i, \tilde y (i)$ as above instead of $x_i, y_i, y(i)$, respectively. With no loss of generality, after passing to a subsequence, we can assume that $y(i), \tilde y(i) \to y, \tilde y$ hold for some $y, \tilde y \in Y$, respectively. Then it is enough to check that $y=\tilde y$.

By the construction done in Theorem \ref{thm:convi}, we have
\begin{align}
\dist_{\ell^2}\left(I^{\mathbf{b}}_t(y), I^{\mathbf{b}}_t(\tilde y)\right)=\lim_{i \to \infty}\dist_{\ell^2}\left(I^{\mathbf{a}_i}_{t_i}(x_i), I^{\mathbf{a}_i}_{t_i}(\tilde x_i)\right) \le C \lim_{i \to \infty}\dist_{X_i}(x_i, \tilde x_i) =0, 
\end{align}
where we used a $C$-Lipschitz continuity of $\Phi_{t_i}^{X_i}$ (coming from (\ref{eq:equi lip})) stated in the beginning of subsection \ref{88shsnsn}, where $C$ is independent of $i$. Thus $y=\tilde y$, namely $f$ is well-defined.
Moreover, by the construction of $f$ (in particular (\ref{hsihasha8ssn})), we see that (\ref{condi}) holds, thus $f$ gives an isometry as metric measured spaces because of Theorem \ref{thmdeg} with $\meas_{X}(X)=\meas_Y(Y)$. 

Let us prove  by contradiction that
$f_i$ uniformly converge to $f$ in the sense (\ref{snsyshsvsgs}). 
If $f_i$ does not converge uniformly to $f$, then there exist a positive constant $\tilde \tau \in (0, 1)$ and sequences of $x_i, \tilde x_i \in X_i$ such that $\dist_{X_i}(x_i, \tilde x_i) \to 0$ and $\dist_{Y_i}(f_i(x_i), f_i(\tilde x_i)) \ge \tilde \tau$ are satisfied. With no loss of generality we can assume that $x_i, \tilde x_i \to x$ for some $x \in X$. Then recalling the definition of $f$, we have 
\begin{equation}
\tilde \tau \le \dist_{Y_i}\left(f_i(x_i), f_i(\tilde x_i)\right) \to \dist_Y \left(f(x), f(x)\right)=0
\end{equation}
which is a contradiction.

In particular $f_i$ gives an $\delta_i$-GH approximation for some $\delta_i \to 0^+$.
Note that the uniform convergence of $f_i$ implies that $(f_i)_{\sharp}\meas_{X_i}(B_r(w_i)) \to f_{\sharp}\meas_X(B_r(w))=\meas_Y(B_r(w))$ for any convergent sequence $w_i \in Y_i \to w \in Y$ with respect to (\ref{100}). From this, 
we can easily get the weak convergence of $(f_i)_{\sharp}\meas_{X_i}$ to $\meas_Y$ with respect to (\ref{100}) (see for instance \cite[Proposition 3.17]{Honda2}).
Thus it holds from \cite[Proposition 4.1 and Corollary 4.3]{LottVillani} that
\begin{equation}
\mathsf{W}^2_{Y_i}\left(\frac{(f_i)_{\sharp}\meas_{X_i}}{\meas_{X_i}(X_i)}, \frac{\meas_{Y_i}}{\meas_{Y_i}(Y_i)}\right) \to \mathsf{W}_Y^2\left(\frac{\meas_Y}{\meas_Y(Y)}, \frac{\meas_Y}{\meas_Y(Y)}\right)=0,
\end{equation}
which proves that $f_i$ gives a $\tilde \delta_i$-mGH-approximation for some $\tilde \delta_i \to 0^+$. This is a contradiction. Thus we have (1).

For (2), we adopt contradiction as the proof again. Recalling that the definition of the limit map $f$ above does not depend on $f_i$, since the proof is similar to the above,  we omit the proof. 
\end{proof}

\subsection{Spectral distance}
Let us fix two compact $\RCD(K, N)$ spaces $\mathbb{X}, \mathbb{Y}$  for some $K \in \mathbb{R}$ and some $N \in [1, \infty)$, and fix $t \in (0, \infty)$.
We introduce a spectral distance, based on the same ideas as \cite{BerardBessonGallot}.
\begin{definition}[Spectral pseudo distance and its variants]\label{as9as9jasnsh}
Let us put
\begin{equation}\label{eq;specinf}
\underline{\dist}_{\mathrm{Spec}}^t\left(\mathbb{X}, \mathbb{Y}\right):=\inf_{\mathbf{a}, \mathbf{b}} \left\{ \dist_{\ell^2}^{\mathrm{H}}\left( I^{\mathbf{a}}_{t}(X), I^{\mathbf{b}}_{t}(Y)\right)\right\},
\end{equation}
\begin{equation}
\stackrel{\to}{\dist^t}_{\mathrm{Spec}}\left(\mathbb{X}, \mathbb{Y}\right):=\sup_{\mathbf{a}} \inf_{\mathbf{b}} \left\{ \dist_{\ell^2}^{\mathrm{H}}\left( I^{\mathbf{a}}_{t}(X), I^{\mathbf{b}}_{t}(Y)\right)\right\},
\end{equation}
and 
\begin{equation}
\dist_{\mathrm{Spec}}^t\left(\mathbb{X}, \mathbb{Y}\right):=\max \left\{ \stackrel{\to}{\dist^t}_{\mathrm{Spec}}\left(\mathbb{X}, \mathbb{Y}\right), \stackrel{\leftarrow}{\dist^t}_{\mathrm{Spec}}\left(\mathbb{X}, \mathbb{Y}\right)\right\},
\end{equation}
where $\mathbf{a}, \mathbf{b}$ above run over all spectral data of $\mathbb{X}, \mathbb{Y}$, respectively, and 
\begin{equation}
\stackrel{\leftarrow}{\dist^t}_{\mathrm{Spec}}\left(\mathbb{X}, \mathbb{Y}\right):= \sup_{\mathbf{b}} \inf_{\mathbf{a}}  \left\{ \dist_{\ell^2}^{\mathrm{H}}\left( I^{\mathbf{a}}_{t}(X), I^{\mathbf{b}}_{t}(Y)\right)\right\}=\stackrel{\to}{\dist^t}_{\mathrm{Spec}}(\mathbb{Y}, \mathbb{X}).
\end{equation}
\end{definition}
It is trivial by definition that
\begin{equation}\label{eq:ineq}
\dist_{\mathrm{Spec}}^t \ge \max \left\{ \stackrel{\to}{\dist^t}_{\mathrm{Spec}}, \stackrel{\leftarrow}{\dist^t}_{\mathrm{Spec}}\right\} \ge \min \left\{ \stackrel{\to}{\dist^t}_{\mathrm{Spec}}, \stackrel{\leftarrow}{\dist^t}_{\mathrm{Spec}}\right\} \ge \underline{\dist}_{\mathrm{Spec}}
\end{equation}
holds and that the infimum in (\ref{eq;specinf}) can be replaced by the minimum, namely
\begin{equation}
\underline{\dist}_{\mathrm{Spec}}^t\left(\mathbb{X}, \mathbb{Y}\right)=\dist_{\ell^2}^H\left( I^{\mathbf{a}}_{t}(X), I^{\mathbf{b}}_t(Y)\right)
\end{equation}
for some spectral data $\mathbf{a}, \mathbf{b}$ of $\mathbb{X}, \mathbb{Y}$, respectively because of applying Theorem \ref{prop:comp} to a fixed space.
As a corollary of this observation with Theorem \ref{thmnond}, we have the following. 
\begin{corollary}\label{cordeg}
$\mathbb{X}$ is isometric to $\mathbb{Y}$ as metric measured spaces, up to a multiplication of $\meas_Y$ by a positive constant, if and only if $\underline{\dist}_{\mathrm{Spec}}^t\left(\mathbb{X}, \mathbb{Y}\right)=0$ holds. In particular we have Corollary \ref{cormain}.
\end{corollary}
\begin{proposition}
For a compact $\RCD(K, N)$ space $\mathbb{Z}$, we have
\begin{align}\label{1}
\stackrel{\to}{\dist^t}_{\mathrm{Spec}}\left(\mathbb{X}, \mathbb{Y}\right) \le \stackrel{\to}{\dist^t}_{\mathrm{Spec}}\left(\mathbb{X}, \mathbb{Z}\right) +\stackrel{\to}{\dist^t}_{\mathrm{Spec}}\left(\mathbb{Z}, \mathbb{Y}\right)
\end{align}
and 
\begin{align}\label{2}
&\stackrel{\leftarrow}{\dist^t}_{\mathrm{Spec}}\left(\mathbb{X}, \mathbb{Y}\right)\le  \stackrel{\leftarrow}{\dist^t}_{\mathrm{Spec}}\left(\mathbb{X}, \mathbb{Z}\right) +\stackrel{\leftarrow}{\dist^t}_{\mathrm{Spec}}\left(\mathbb{Z}, \mathbb{Y}\right).
\end{align}
In particular, combining the above with Corollary \ref{cordeg},  we see that $\dist_{\mathrm{Spec}}^t$ defines a distance on the set of all isometry classes as metric measured spaces of compact $\RCD(K, N)$ spaces with total measures $1$.
\end{proposition}
\begin{proof}
Since (\ref{2}) is a corollary of (\ref{1}), let us prove (\ref{1}) only.
The triangle inequality for the Hausdorff distance yields
\begin{align}
\dist_{\ell^2}^{\mathrm{H}}\left( I^{\mathbf{a}}_{t}(X), I^{\mathbf{b}}_{t}(Y)\right) 
&\le \dist_{\ell^2}^{\mathrm{H}}\left( I^{\mathbf{a}}_{t}(X), I^{\mathbf{c}}_{t}(Z)\right) + \dist_{\ell^2}^{\mathrm{H}}\left( I^{\mathbf{c}}_{t}(Z), I^{\mathbf{b}}_{t}(Y)\right)
\end{align}
for all spectral data $\mathbf{a}, \mathbf{b}, \mathbf{c}$. 

Thus taking the infimum with respect to $\mathbf{b}$ shows
\begin{align}
\inf_{\mathbf{b}} \dist_{\ell^2}^{\mathrm{H}}\left( I^{\mathbf{a}}_{t}(X), I^{\mathbf{b}}_{t}(Y)\right) &\le \dist_{\ell^2}^{\mathrm{H}}\left( I^{\mathbf{a}}_{t}(X), I^{\mathbf{c}}_{t}(Z)\right) +\inf_{\mathbf{b}}\dist_{\ell^2}^{\mathrm{H}}\left( I^{\mathbf{c}}_{t}(Z), I^{\mathbf{b}}_{t}(Y)\right) \nonumber \\
& \le  \dist_{\ell^2}^{\mathrm{H}}\left( I^{\mathbf{a}}_{t}(X), I^{\mathbf{c}}_{t}(Z)\right) +\stackrel{\to}{\dist^t}_{\mathrm{Spec}}\left(\mathbb{Z}, \mathbb{Y}\right).
\end{align}
Moreover taking the infimum with respect to $\mathbf{c}$ and then taking the supremum with respect to $\mathbf{a}$ complete the proof of (\ref{1}). 

\end{proof}

\section{$\dist_{\mathrm{mGH}}$ vs $\dist_{\mathrm{Spec}}^t$}
The main purpose of this section is to provide a proof of Theorem \ref{mainthm2}. The following is a generalization of \cite[(ii) of Theorem 17]{BerardBessonGallot} to the RCD setting.
\begin{theorem}\label{eq:mghspec}
Let 
\begin{equation}
\mathbb{X}_i \stackrel{\mathrm{mGH}}{\to} \mathbb{X} \quad \text{in}\,\,\,\mathcal{M}(K, N, d, v)
\end{equation}
for some $K, N, d, v$. Then for any $t \in (0, \infty)$, after passing to a subsequence, we have
\begin{equation}
\stackrel{\to}{\dist^t}_{\mathrm{Spec}}(\mathbb{X}_i, \mathbb{X}) \to 0.
\end{equation}
In particular $\underline{\dist^t}_{\mathrm{Spec}}(\mathbb{X}_i, \mathbb{X}) \to 0$.
\end{theorem}
\begin{proof}
Take a sequence of spectral datas $\mathbf{a}_i$ of $\mathbb{X}_i$ with
\begin{equation}
\left| \stackrel{\to}{\dist^t}_{\mathrm{Spec}}(\mathbb{X}_i, \mathbb{X}) - \inf_{\mathbf{c}}\dist^{\mathrm{H}}_{\ell^2}\left(I^{\mathbf{a}_i}_{t}(X_i), I^{\mathbf{c}}_{t}(X)\right) \right| \to 0.
\end{equation}
Thanks to Theorem \ref{prop:comp}, after passing to a subsequence, we have $\mathbf{a}_i\to \mathbf{a}$ for some spectral datum $\mathbf{a}$ of $\mathbb{X}$. 
Then since
\begin{equation}
\dist^{\mathrm{H}}_{\ell^2}\left(I^{\mathbf{a}_i}_{t}(X_i), I^{\mathbf{a}}_{t}(X)\right)+ \left| \stackrel{\to}{\dist^t}_{\mathrm{Spec}}(\mathbb{X}_i, \mathbb{X}) - \inf_{\mathbf{c}}\dist^{\mathrm{H}}_{\ell^2}\left(I^{\mathbf{a}_i}_{t}(X_i), I^{\mathbf{c}}_{t}(X)\right) \right|  \ge \stackrel{\to}{\dist^t}_{\mathrm{Spec}}(\mathbb{X}_i, \mathbb{X})
\end{equation}
and the LHS converge to $0$ because of Theorem \ref{thm:convi}, we conclude.
\end{proof}
\begin{theorem}\label{thmoneimp}
Let $\mathbb{X}_i$ be a sequence of compact $\RCD(K, N)$ spaces in $\mathcal{M}=\mathcal{M}(K, N, d, v)$ for some $K, N, d, v$ and let $\mathbb{X} \in \mathcal{M}$.  If
\begin{equation}
\meas_{X_i}(X_i) \to \meas_X(X)
\end{equation}
and 
\begin{equation}
\dist_{\mathrm{Spec}}^t(\mathbb{X}_i, \mathbb{X}) \to 0
\end{equation}
hold for some $t \in (0, \infty)$, then we have
\begin{equation}
\mathbb{X}_i \stackrel{\mathrm{mGH}}{\to} \mathbb{X}.
\end{equation}
\end{theorem}
\begin{proof}
The proof is done by contradition. If not, then after passing to a subsequence, there exists $\tilde{\mathbb{X}} \in \mathcal{M}$ such that $\meas_{\tilde X}(\tilde X)=\meas_X(X)$ holds, that
\begin{equation}
\mathbb{X}_i \stackrel{\mathrm{mGH}}{\to} \tilde{\mathbb{X}}
\end{equation} 
holds and that $\mathbb{X}$ is not isometric to $\tilde{\mathbb{X}}$ as metric measured spaces.
Thus Theorem \ref{eq:mghspec} with (\ref{eq:ineq}) and (\ref{1}) shows
\begin{align}
\underline{\dist}_{\mathrm{Spec}}^t(\mathbb{X}, \tilde{\mathbb{X}}) \le \stackrel{\to}{\dist^t}_{\mathrm{Spec}}(\mathbb{X}, \tilde{\mathbb{X}})&\le   \stackrel{\to}{\dist^t}_{\mathrm{Spec}}\left(\mathbb{X}, \mathbb{X}_i\right) +\stackrel{\to}{\dist^t}_{\mathrm{Spec}}\left(\mathbb{X}_i, \tilde{\mathbb{X}}\right)\nonumber \\
&\le \dist_{\mathrm{Spec}}^t\left(\mathbb{X}, \mathbb{X}_i\right) +\stackrel{\to}{\dist^t}_{\mathrm{Spec}}\left(\mathbb{X}_i, \tilde{\mathbb{X}}\right) \to 0.
\end{align}
Therefore Corollary \ref{cordeg} allows us to conclude that $\mathbb{X}$ is isometric to $\tilde{\mathbb{X}}$ as metric measured spaces, which is a contradiction.
\end{proof}
We are now in a position to prove Theorem \ref{mainthm2}. 
\begin{proof}[Proof of Theorem \ref{mainthm2}]
Firstly we omit a result, Corollary \ref{aaassrr}, which will be proved in subsection \ref{result}, to get 
\begin{equation}
\liminf_{i \to \infty}\mathrm{diam}(X_i, \dist_{X_i})>0
\end{equation}
under any of conditions (1), (2), (3) and (4). Namely it is enough to focus only on sequences in $\mathcal{M}(K, N, d, v)$ for some $d, v$.

Recall that in Corollary \ref{cordef} we already proved the equivalence between (3) and (4), thus
firstly let us prove the implication from (1) to (3). 

\textit{The proof of $(1) \Rightarrow (3)$.}

Assume that (1) holds. Then the desired mGH convergence;
\begin{equation}\label{102}
\mathbb{X}_i \xrightarrow[f_i]{\mathrm{mGH}} \mathbb{X}
\end{equation}
is already established by Theorem \ref{thmoneimp}.
Thus we focus on the remaining convergence results $\mu_j(\mathbb{X}_i) \to \mu_j(\mathbb{X})$ and $\nu_j(\mathbb{X}_i) \to \nu_j(\mathbb{X})$ for any $j$.

As discussed in the proof of Corollary \ref{cordef}, we have $\mu_j(\mathbb{X}_i) \to \mu_j(\mathbb{X})$ for any $j \le 1$.
Assume that  $\nu_1(\mathbb{X}_i)$ does not converge to $\nu_1(\mathbb{X})$. 
After passing to a subsequence, with no loss of generality we can assume that $\nu_1(\mathbb{X}_i)$ is a constant $k$ which is independent of $i$. Then Theorem \ref{prop:comp} easily allows us to conclude that $k=\nu_1(\mathbb{X}_i) < \nu_1(\mathbb{X})$ and $\mu_2(\mathbb{X}_i) \to \mu_1(\mathbb{X})$ are satisfied, namely a finite sequence of eigenvalues of the operator $-\Delta_{X_i}$;
\begin{equation}
0=\mu_0(\mathbb{X}_i)<\underbrace{\mu_1(\mathbb{X}_i) = \cdots =\mu_1(\mathbb{X}_i)}_{k=\nu_1(\mathbb{X}_i)} <\mu_2(\mathbb{X}_i)
\end{equation}
``converge'' to a finite sequence of eigenvalues of the operator $-\Delta_X$;
\begin{equation}
0=\mu_0(\mathbb{X})<\underbrace{\mu_1(\mathbb{X}) = \cdots =\mu_1(\mathbb{X})}_{k} =\mu_1(\mathbb{X}),
\end{equation}
respectively.
Finding a sequence of spectral data $\mathbf{a}_i$ of $\mathbb{X}_i$ converging to a spectral datum $\mathbf{a}$ of $\mathbb{X}$, define a spectral datum $\mathbf{b}$ of $\mathbb{X}$ by
\begin{equation}
\mathbf{b}:= \left( \phi_0^{\mathbf{a}}, \ldots, \phi_{k-1}^{\mathbf{a}}, \frac{1}{\sqrt{2}}\left( \phi_k^{\mathbf{a}}+\phi_{k+1}^{\mathbf{a}}\right), \frac{1}{\sqrt{2}}\left( \phi_k^{\mathbf{a}}-\phi_{k+1}^{\mathbf{a}}\right), \phi_{k+2}^{\mathbf{a}}, \ldots \right).
\end{equation}
Then our assumption (\ref{101}) yields that there exists a sequence of spectral data $\mathbf{b}_i$ of $\mathbb{X}_i$ such that 
\begin{equation}
\dist_{\ell^2}^{\mathrm{H}}\left(I^{\mathbf{b}_i}_t(X_i), I^{\mathbf{b}}_t(X)\right) \to 0
\end{equation}
holds.

On the other hand, by the very definition of $\nu_1(\mathbb{X}_i)$ there exists $A_i \in O(k)$ such that 
\begin{equation}
\mathbf{b}_i=\left( \phi_0^{\mathbf{a}_i}, \left( \phi_1^{\mathbf{a}_i}, \ldots, \phi_k^{\mathbf{a}_i}\right)A_i, \phi_{k+1}^{\mathbf{a}_i},\ldots\right) 
\end{equation}
holds. Thanks to the compactness of $O(k)$, with no loss of generality we can assume that $A_i \to A$ holds for some  $A \in O(k)$.  In particular $\mathbf{b}$ can be written by
\begin{equation}
\mathbf{b}=\left( \phi_0^{\mathbf{a}}, \left( \phi_1^{\mathbf{a}}, \ldots, \phi_k^{\mathbf{a}}\right)A, \phi_{k+1}^{\mathbf{a}},\ldots\right), 
\end{equation}
which shows that $\phi_{k+1}^{\mathbf{a}}$ can be obtained by a linear combination of $\phi_1^{\mathbf{a}}, \ldots, \phi_k^{\mathbf{a}}$. This is a contradiction. Thus we have $\nu_1(\mathbb{X}_i) \to \nu_1(\mathbb{X})$. Moreover this observation also proves $\mu_2(\mathbb{X}_i) \to \mu_2(\mathbb{X})$.

Since other cases, $j\ge 2$, is similarly done  by induction on $j$, we omit the proof. Thus we have (3).

Next let us give a proof of the remaining nontrivial implication, from (3) to (2). A key (but trivial) observation from linear algebra is to find an orthonormal transformation, written as $B_j$ below, between given two orthonormal basis of the eigenspace $E_{\lambda}(\mathbb{X}_i)=\{f \in D(\Delta_{X_i})| \Delta_{X_i}f+\lambda f=0\}$ of any eigenvalue $\lambda$ of the operator $-\Delta_{X_i}$.

\textit{The proof of $(3) \Rightarrow (2)$.}

Assume that (3) holds (under the same notation as in (\ref{102})).
Let us fix a convergent sequence of spectral data $\mathbf{a}_i, \to \mathbf{a}$. Take a spectral datum $\mathbf{b}$ of $\mathbb{X}$. For any $m \in \mathbb{N}$, define a new sequence of spectral data $\mathbf{b}(m)$ of $\mathbb{X}$ by  
\begin{equation}
\mathbf{b}(m):=\left(\phi_0^{\mathbf{b}},\ldots, \phi_{b_m}^{\mathbf{b}}, \phi_{b_{m}+1}^{\mathbf{a}},\ldots \right),\quad \text{where}\,\,\,b_m:=\sum_{i=1}^m\nu_i(\mathbb{X}).
\end{equation}
Note that $\mathbf{b}(m) \to \mathbf{b}$ holds (with respect to the trivial convergence $\mathbb{X} \xrightarrow[\mathrm{id}_X]{\mathrm{mGH}} \mathbb{X}$) because of (\ref{eq:eigenfunction}). Therefore, recalling Definition \ref{as9as9jasnsh} with Proposition \ref{prop:rela}, in order to prove (2), it is enough to prove that for any $m \in \mathbb{N}$ there exists a sequence of spectral data $\mathbf{b}_n(m)$ of $\mathbb{X}_n$ such that $\mathbf{b}_n(m) \to \mathbf{b}(m)$ with respect to (\ref{102}). 

Actually this is done by putting
\begin{equation}
\mathbf{b}_n(m):=\left(\phi_0^{\mathbf{a}_i}, \left(\phi_1^{\mathbf{a}_n},\ldots,\phi_{\nu_1(\mathbb{X})}^{\mathbf{a}_n}\right)B_1,\ldots, \left(\phi_{b_{m-1}+1}^{\mathbf{a}_n},\ldots,\phi_{b_m}^{\mathbf{a}_n}\right)B_m, \phi_{b_m+1}^{\mathbf{a}_n},\ldots\right),
\end{equation}
where this is well-defined for any sufficiently large $n$, and $B_i \in O(\nu_i(\mathbb{X}))$ is defined by satisfying
\begin{equation}
\left(\underbrace{\phi_{b_{i-1}+1}^{\mathbf{b}}, \ldots, \phi_{b_i}^{\mathbf{b}}}_{\nu_i(\mathbb{X})}\right)=\left(\underbrace{\phi_{b_{i-1}+1}^{\mathbf{a}},\ldots,\phi_{b_i}^{\mathbf{a}}}_{\nu_i(\mathbb{X})}\right)B_i.
\end{equation}
Thus we have the desired convergence $\mathbf{b}_n(m) \to \mathbf{b}(m)$. Therefore we have  (2).
\end{proof}
Recall that \cite[(i) of Theorem 17]{BerardBessonGallot} proves the $\dist_{\mathrm{Spec}}^t$-convergence for a Lipschitz convergent sequence of closed Riemannian manifolds with Ricci curvature bounded below under assuming that the limit space has simple spectrum. Note that if the limit space of a mGH-convergent sequence has simple spectrum, then we can easily check (\ref{asioraoiha}) because of (\ref{444}). Therefore Theorem \ref{mainthm2} also recovers this result. 

Finally let us give an improvement of Theorem \ref{mainthm2} for a special class of $\RCD$ spaces, so-called noncollapsed $\RCD$ spaces, where an $\RCD(K, N)$ space is said to be \textit{noncollapsed} if the reference measure coincides with the $N$-dimensional Hausdorff measure $\mathcal{H}^N$. See \cite{DG} for the details.
\begin{corollary}\label{noncollapsed}
Let $\mathbb{X}_i, \mathbb{X}$ be compact noncollapsed $\RCD(K, N)$ spaces $(i=1,2,\ldots)$. Assume $\sup_i\mathrm{diam}(X_{i}, \dist_{X_i})<\infty$. If
\begin{equation}
\dist_{\mathrm{Spec}}^t\left(\mathbb{X}_i, \mathbb{X}\right) \to 0
\end{equation}
holds for some $t \in (0, \infty)$, then $\mathcal{H}^N(X_i) \to \mathcal{H}^N(X)$. In particular we have
\begin{equation}
\mathbb{X}_i \stackrel{\mathrm{mGH}}{\to} \mathbb{X},\quad \mu_j(\mathbb{X}_i) \to \mu_j(\mathbb{X})\,\,\,\text{and}\,\,\,\nu_j(\mathbb{X}_i) \to \nu_j(\mathbb{X}), \quad \forall j.
\end{equation}
\end{corollary}
\begin{proof}
Let 
\begin{equation}
\tilde{\mathbb{X}}_i:=\left(X_i, \dist_{X_i}, \frac{\mathcal{H}^N}{\mathcal{H}^N(X_i)}\right), \quad \tilde{\mathbb{X}}:= \left(X, \dist_X,  \frac{\mathcal{H}^N}{\mathcal{H}^N(X)}\right).
\end{equation}
Then since $\dist_{\mathrm{Spec}}^t(\tilde{\mathbb{X}}_i, \tilde{\mathbb{X}}) \to 0$ because of Remark \ref{scalei}, 
Theorem \ref{mainthm2} yields
\begin{equation}
\tilde{\mathbb{X}}_i \stackrel{\mathrm{mGH}}{\to}\tilde{\mathbb{X}}.
\end{equation}
Thus it follows from \cite[Theorem 1.2]{DG} that $\mathcal{H}^N(X_i) \to \mathcal{H}^N(X)$ holds. The remaining statements also come from Theorem \ref{mainthm2}.
\end{proof}
\section{Spectral distance in the sense of  Kasue-Kumura}\label{77889}
Before giving the proof of Theorem \ref{kk}, let us provide  the definition of the spectral distance in the sense of \cite{KK, KK2} for compact $\RCD(K, N)$ spaces.
\begin{definition}[Spectral distance in the sense of  Kasue-Kumura]
Let $\mathbb{X}, \mathbb{Y}$ be compact $\RCD(K, N)$ spaces for some $K \in \mathbb{R}$ and some $N \in [1, \infty)$. Then the \textit{spectral distance} in the sense of Kasue-Kumura between them, denoted by $\tilde \dist_{\mathrm{Spec}}(\mathbb{X}, \mathbb{Y})$, is defined by the infimum of $\epsilon \in (0, \infty)$ satisfying that there exist maps $f:X \to Y$ and $g:Y \to X$ such that 
\begin{equation}
e^{-(t+1/t)}\left|p_Y(f(x), f(\tilde x), t)-p_X(x, \tilde x, t)\right|<\epsilon,\quad \forall x,\,\,\forall \tilde x \in X,\,\,\,\forall t \in (0, \infty)
\end{equation}
and 
\begin{equation}
e^{-(t+1/t)}\left|p_X(g(y), g(\tilde y), t)-p_Y(y, \tilde y, t)\right|<\epsilon,\quad \forall y,\,\,\forall\tilde y \in Y,\,\,\,\forall t \in (0, \infty)
\end{equation}
are satisfied.
\end{definition}
The proof of the next proposition is given by the same ideas as discussed in \cite{KK} (see also \cite{KK2}) in the smooth framework. For reader's convenience, let us provide the proof in this setting.
\begin{proposition}\label{prosss}
The spectral distance $\tilde \dist_{\mathrm{Spec}}$ in the sense above is indeed a distance on $\mathcal{M}=\mathcal{M}(K, N, d, v)$ for all $K, N, d, v$.
\end{proposition}
\begin{proof}
It is easily checked by definition that $\tilde \dist_{\mathrm{Spec}}$ is symmetric with the triangle inequality. Thus let us prove the nondegeneracy. Assume that $\tilde \dist_{\mathrm{Spec}}(\mathbb{X}, \mathbb{Y})=0$ holds for some $\mathbb{X}, \mathbb{Y} \in \mathcal{M}$. Then there exist sequences of maps $f_i:X \to Y$ and $g_i:Y \to X$ such that 
\begin{equation}
e^{-(t+1/t)}\left|p_Y(f_i(x), f_i(\tilde x), t)-p_X(x, \tilde x, t)\right|<\frac{1}{i},\quad \forall x,\,\,\forall \tilde x \in X,\,\,\,\forall t \in (0, \infty)
\end{equation}
and 
\begin{equation}
e^{-(t+1/t)}\left|p_X(g_i(y), g_i(\tilde y), t)-p_Y(y, \tilde y, t)\right|<\frac{1}{i},\quad \forall y,\,\,\forall\tilde y \in Y,\,\,\,\forall t \in (0, \infty)
\end{equation}
are satisfied. Let $A$ be a countable dense subset of $X$. After passing to a subsequence with a diagonal argument, we can assume that $\{f_i(x)\}_i$ is a convergent sequence in $Y$ for any $x \in A$. Then defining a map $f:A \to Y$ by $f(x):=\lim_{i \to \infty}f_i(x)$, we have $p_Y(f(x), f(\tilde x), t)=p_X(x, \tilde x, t)$ for all $x, \tilde x \in A$ and $t \in (0, \infty)$. In particular (\ref{varad}) allows us to conclude that $f$ preserves the distances. Thus there exists a unique continuous extention of $f$ to a map from $X$ to $Y$, still denoted by $f:X \to Y$. It is trivial that $f$ also preserves the distance and the heat kernel. In particular $f$ is injective.
Similarly we can construct a map $g:Y \to X$ which preserves the distance (and the heat kernel).

On the other hand, since $f \circ g:Y \to Y$ preserves the distance, it must be an isometry as metric spaces  because $Y$ is compact  (see for instance \cite{BBI}). In particular $f$ is surjective. Thus $f$ is an isometry as metric spaces.

In order to check that $f$ preserves the measures, taking $\psi \in C(Y)$, we have
\begin{align}
\int_Y\psi \di \left(f_{\sharp}\meas_X\right)&=\int_X\psi \circ f(x)\di \meas_X(x) \nonumber \\
&=\lim_{t \to 0^+}\int_X \int_Y\psi (y)p_Y(y, f(x), t)\di \meas_Y(y)\di \meas_X(x) \nonumber \\
&=\lim_{t \to 0^+}\int_Y \int_X\psi (y)p_X(f^{-1}(y), x, t)\di \meas_X(x)\di \meas_Y(y) =\int_Y\psi \di \meas_Y.
\end{align}
Thus $f_{\sharp}\meas_X=\meas_Y$, namely $\mathbb{X}$ is isometric to $\mathbb{Y}$ as metric measured spaces.
\end{proof}
Let us end this section by finishing the proof of the remaing main result. 
\begin{proof}[Proof of Theorem \ref{kk}]
As in the proof of Theorem \ref{mainthm2}, we also omit a result, Proposition \ref{proplower}. The proposition allows us to show that the following argument is enough to conclude.

Recalling that $\mathcal{M}=\mathcal{M}(K, N, d, v)$ is compact with respect to $\dist_{\mathrm{mGH}}$,  thanks to Proposition \ref{prosss},  it is enough to check that the canonical map, $\mathrm{id}: (\mathcal{M}, \dist_{\mathrm{mGH}}) \to (\mathcal{M}, \tilde \dist_{\mathrm{Spec}})$, is continuous. Actually this continuity is achieved by applying the pointwise convergence of the heat kernels with respect to mGH-convergence proved in \cite[Theorem 3.3]{AmbrosioHondaTewodrose} (see also \cite[Corollary 3.10]{ZZ}) with Theorem \ref{thm:gaussian}. Thus we conclude. 
\end{proof}
\section{Collapsing to a single point}\label{finalfinal}
In this final section we discuss the remaining case; a sequence of compact $\RCD(K, N)$ spaces converge to a single point.
Firstly let us clarify the meanings of $I^{\mathbf{a}}_t$ and of the heat kernel on a single point.
\subsection{Convention}
Let $\mathbb{X}$ be a metric measured space and assume that this is a single point, namely $X=\{x\}$ and $\meas_X=c\delta_x$ for some $c \in (0, \infty)$, where $\delta_x$ denotes the Dirac measure at $x$.
Then the heat kernel $p_X$ of $\mathbb{X}$ is defined by
\begin{equation}
p_X(x,x, t):=\frac{1}{\meas_X(X)}=\frac{1}{c}.
\end{equation}
Furthermore the spectral datum $\mathbf{a}$ of $\mathbb{X}$ is uniquely determined by
\begin{equation}
\mathbf{a}=\left(\frac{1}{\sqrt{c}},0,0,\ldots \right)
\end{equation}
and the embedding $I^{\mathbf{a}}_t:X \to \ell^2$ is defined by
\begin{equation}
I^{\mathbf{a}}_t(x):=\left(0,0,\ldots \right) \in \ell^2.
\end{equation}
Then we can consider the previous notions even in this setting, namely, for example, the spectral distances in the sense of \cite{BerardBessonGallot} or of \cite{KK} between compact $\RCD(K, N)$ spaces which are possibly single points, are well-defined by the same ways. 
\subsection{Results}\label{result}
Let us emphasize that Theorems \ref{asjaisahrah} and \ref{asoariiiis} below show that $\dist_{\mathrm{Spec}}^t$ and $\dist_{\mathrm{mGH}}$ are equivalent \textit{without any assumption} on eigenvalues if the limit space is a single point.
\begin{theorem}\label{asjaisahrah}
Let 
\begin{equation}\label{9as0a9r8a}
\mathbb{X}_i \stackrel{\mathrm{mGH}}{\to} \mathbb{X}.
\end{equation}
be a mGH convergent sequence of compact $\RCD(K, N)$ spaces for some $K \in \mathbb{R}$ and some $N \in [1, \infty)$. Assume that $\mathbb{X}$ is a single point.
Then for any $t \in (0, \infty)$ we have 
\begin{equation}\label{asohaorha}
\dist_{\mathrm{Spec}}^t(\mathbb{X}_i, \mathbb{X}) \to 0.
\end{equation}
\end{theorem}
\begin{proof}
Thanks to (\ref{eq:eigenfunction}), we have for any sufficiently large $i$,
\begin{equation}
|e^{-\lambda_j(\mathbb{X}_i)t/2}\phi_j^{X_i}| \le C(K,N)e^{-\lambda_j(\mathbb{X}_i)t/2}\lambda_j(\mathbb{X}_i)^{N/4},\quad \forall j.
\end{equation}
Conbining this with a fact that $\lambda_1(\mathbb{X}_i) \to \infty$ implies
\begin{equation}\label{eq56g}
\dist^{\mathrm{H}}_{\ell^2}(I^{\mathbf{a}_i}_t(X_i), I^{\mathbf{a}}_t(X)) \to 0
\end{equation}
for all spectral data $\mathbf{a}_i, \mathbf{a}$ of $\mathbb{X}_i, \mathbb{X}$, respectively. In particular (\ref{eq56g}) shows (\ref{asohaorha}).
\end{proof}
\begin{corollary}\label{aaassrr}
Let $\mathbb{X}_i, \mathbb{X}$ be compact $\RCD(K, N)$ spaces $(i=1,2,\ldots)$. If $\mathbb{X}$ is not a single point and (\ref{asohaorha}) holds for some $t \in (0, \infty)$, then
\begin{equation}
\liminf_{i \to \infty}\mathrm{diam}(X_i, \dist_i)>0.
\end{equation}
\end{corollary}
\begin{proof}
The proof is done by contradiction. If the assertion is not satisfied, then after passing to a subsequence with normalizations of measures, we can assume that $\mathbb{X}_i$ mGH converge to a single point $\tilde{\mathbb{X}}$.
Thus since Theorem \ref{asjaisahrah} yields
$\dist_{\mathrm{Spec}}^t(\mathbb{X}, \tilde{\mathbb{X}}) = 0$, we see that $\mathbb{X}$ must be also a single point.
This is a contradiction.
\end{proof}
\begin{theorem}\label{asoariiiis}
Let $\mathbb{X}_i$ be a sequence of compact $\RCD(K, N)$ spaces for some $K \in \mathbb{R}$ and some $N \in [1, \infty)$. Assume that $\sup_i\mathrm{diam}(X_i, \dist_{X_i})<\infty$ holds, that $\meas_{X_i}(X_i) \to c$ holds for some $c \in (0, \infty)$ and that
(\ref{asohaorha})
holds for a single point $\mathbb{X}$ and some $t \in (0, \infty)$. Then (\ref{9as0a9r8a}) holds with $\meas_X(X)=c$.
\end{theorem}
\begin{proof}
The proof is also done by a contradiction. If the assertion is not satisfied, then after passing to a subsequence, we have
\begin{equation}\label{9as0a9r8aaa}
\mathbb{X}_i \stackrel{\mathrm{mGH}}{\to} \tilde{\mathbb{X}}
\end{equation}
for some compact $\RCD(K, N)$ space $\tilde{\mathbb{X}}$ with $\meas_{\tilde X}(\tilde X)=c$, which is not a single point.
On the other hand (\ref{asohaorha}) easily yields that $I^{\mathbf{b}}_t(\tilde X)$ is a single point for any spectral datum $\mathbf{b}$ of $\tilde{\mathbb{X}}$.  This is a contradiction.
\end{proof}




Finally by similar ways as above, we have the following, where we omit the proof.
\begin{proposition}\label{proplower}
Let $\mathbb{X}_i, \mathbb{X}$ be compact $\RCD(K, N)$ spaces $(i=1,2,\ldots)$. Assume that $\sup_i\mathrm{diam}(X_i, \dist_{X_i})<\infty$ holds, that $\mathbb{X}$ is not a single point and that $\mathbb{X}_i$ $\tilde{\dist}_{\mathrm{Spec}}$-converge to $\mathbb{X}$. Then
\begin{equation}
\liminf_{i \to \infty}\mathrm{diam}(X_i, \dist_i)>0.
\end{equation}
\end{proposition}
\begin{theorem} Theorem \ref{kk} is still satisfied if $\mathbb{X}$ is a single point. 
\end{theorem}

\end{document}